\documentclass[12pt]{amsart}
\usepackage[latin1]{inputenc}
\usepackage{amsmath} 
\usepackage{amsfonts}
\usepackage{amssymb}
\usepackage{stmaryrd}
\usepackage{latexsym} 
\usepackage{graphicx}
\usepackage{subfigure}
\usepackage{color}
\usepackage{hyperref}
\usepackage{verbatim}
\usepackage[all]{xy}
\usepackage{graphics}
\usepackage{pdfsync}

\DeclareMathOperator{\Dec}{Dec}

\usepackage{tikz}

\theoremstyle{plain}
\newtheorem{theorem}{Theorem}[section]
\newtheorem{lemma}[theorem]{Lemma}
\newtheorem{corollary}[theorem]{Corollary}

\newtheorem{definition}[theorem]{Definition}
\newtheorem{example}[theorem]{Example}

\theoremstyle{remark}
\newtheorem{remark}[theorem]{Remark}

\title[Sublattices of free lattices]{Generalizing Galvin and J\'onsson's classification to \textbf{N}$_\textbf{5}$}

\author{Brian T. Chan}
\address{
 Department of Mathematics \\
 University of British Columbia\\
 Vancouver BC V6T 1Z2, Canada}
\email{bchan@math.ubc.ca}

\date{\today}
\thanks{The author was supported in part by the Natural Sciences and Engineering Research Council of Canada \includegraphics[scale = 0.2]{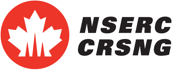} [funding reference number PGSD2 - 519022 - 2018]}

\begin{document}

\begin{abstract} The problem of determining (up to lattice isomorphism) which lattices are sublattices of free lattices is in general an extremely difficult and an unsolved problem. A notable result towards solving this problem was established by Galvin and J\'onsson when they classified (up to lattice isomorphism) all of the distributive sublattices of free lattices in 1959. In this paper, we weaken the requirement that a sublattice of a free lattice be distributive to requiring that a such a lattice belongs in the variety of lattices generated by the pentagon $N_5$. Specifically, we use McKenzie's list of join-irreducible covers of the variety generated by $N_5$ to extend Galvin and J\'onsson's results by proving that all sublattices of a free lattice that belong to the variety generated by $N_5$ satisfy three structural properties. Afterwards, we explain how the results in this paper can be partially extended to lattices from seven known infinite sequences of semidistributive lattice varieties.
\end{abstract}
\maketitle
%
\section{Introduction} \label{sec:intro}

Free lattices have been the subject of much investigation within lattice theory, with Whitman introducing \emph{Whitman's condition} \cite{Whitman 1, Whitman 2} and J\'onsson introducing \emph{semidistributive lattices} to study properties of free lattices \cite{Semidistributive origins 1, Semidistributive origins 2}. An important, and far from solved, problem within the theory of free lattices that has received a lot of attention over the years is the problem of determining, up to lattice isomorphism, sublattices of free lattices \cite{FL}. The majority of what is known about sublattices of free lattices is based on what we know about finite sublattices of free lattices, and includes extensions to finitely generated sublattices of free lattices and projective lattices \cite{FL}. Finite sublattices of free lattices can be characterized by using Whitman's condition and a property involving \emph{join covers} of elements \cite{FL}. Later on, this characterization was strenghened to requiring only the semidistributive laws and Whitman's condition \cite{Nations proof}. Regarding properties that are satisfied by all sublattices of free lattices, the following is known. In 1982, Baldwin, Berman, Glass, and Hodges \cite{Baldwin et al} proved that if $S$ is an uncountable antichain in a free lattice, then $|\{a \wedge b : a, b \in S \text{ and } a \neq b \} | > 1$. Moreover, in 1995,  Reinhold \cite{Reinhold} proved that all sublattices of free lattices satisfy stronger forms of the semidistributive laws known as the \emph{staircase distributive law} and the \emph{dual staircase distributive law} by proving that all free lattices satisfy the \emph{$*$-distributive laws}, an infinitary form of the staircase distributive laws \cite{Reinhold}. \\

A notable result towards analysing sublattices of free lattices was proved in 1959 by Galvin and J\'onsson when they classified, up to lattice isomorphism, all of the distributive sublattices of free lattices \cite{Distributive sublattices}. As the variety of distributive lattices is the smallest variety of lattices, a natural question to ask is whether Galvin and J\'onsson's results can be extended to other, more general, varieties of lattices. In this paper, we consider the second smallest variety of lattices $\mathcal{N}$ that contains sublattices of free lattices, where $\mathcal{N}$ is the variety of lattices generated by $N_5$, and prove three structural properties that all such sublattices satisfy. The first property, Theorem \ref{cube theorem}, is related to Galvin and J\'onsson's classification and involves atoms and coatoms of lattices. The second and third properties, Theorem \ref{partial modularity 1} and Theorem \ref{partial modularity 2}, resemble properties satisfied by modular lattices. Afterwards, we explain in Corollary \ref{c1}, Corollary \ref{c2}, Corollary \ref{c3}, Corollary \ref{c4}, and Corollary \ref{c5}, how of the results in this paper can be partially extended to lattices from seven known infinite sequences of semidistributive lattice varieties. \\

The paper is structured as follows. In Section \ref{sec:background}, we give the background which includes relevant results relating to the variety generated by the pentagon $N_5$, and Galvin and J\'onsson's classification of distributive sublattices of free lattices. In Section \ref{sec:atoms and coatoms} and Section \ref{sec:perpective properties}, we prove the main results of this paper which are three structural properties that are satisfied by all sublattice of a free lattice that are in the variety generated by the pentagon. Lastly, in Section \ref{sec:extensions}, we describe how the results in Section \ref{sec:atoms and coatoms} and Section \ref{sec:perpective properties} can be partially extended to lattices from seven known infinite sequences of semidistributive lattice varieties.

\section{Background}\label{sec:background}

Let $\mathbb{N}$ denote the set of positive integers, and let $\mathbb{N}_0$ denote the set of non-negative integers. If $P$ is a poset and if $a,b \in P$, then we write $a \parallel b$ to mean that $a \leq b$ is false and that $b \leq a$ is false. We also write $a \geq b$ to mean that $b \leq a$, $a < b$ to mean that $a \leq b$ and $a \neq b$, $a > b$ to mean that $b < a$, $a \nless b$ to mean that $a < b$ is false, and $a \ngtr b$ to mean that $a > b$ is false. If $P$ is a poset, then we consider any subset of $P$ as a subposet with partial order inherited from $P$ and vice versa. If $S$ is a set, then a \emph{set partition} of $S$ is a set $\mathcal{F}$ of non-empty subsets of $S$ such that every element of $S$ is contained in exactly one element of $\mathcal{F}$. A subset $S$ of a poset $P$ is \emph{convex} if for all $a, b, c \in L$ such that $a,b \in S$, $a \leq c \leq b$ implies that $c \in S$. We call convex subsets convex subposets and vice versa. If $A$ and $B$ are subsets of a poset $P$, we write $A \cup B$ to denote the subposet of $P$ whose set of elements is the set-theoretic union of $A$ and $B$ as sets, and we write $A \cap B$ to denote the subposet of $P$ whose set of elements is the set-theoretic intersection of $A$ and $B$ as sets. Lastly, if $P$ is a poset, if $a,b \in P$, and if $a < b$, then \emph{$b$ covers $a$ in $P$} (or \emph{$a$ is covered by $b$ in $P$}) if, for all $c \in P$ satisfying $a \leq c \leq b$, $c = a$ or $c = b$. \\

We write $a \vee b$ to denote joins in a lattice and we write $a \wedge b$ to denote meets in a lattice. Recall that a subposet $K$ of a lattice $L$ is a \emph{sublattice of $L$} if, for all $a, b \in K$, $a \vee b \in K$ and $a \wedge b \in K$. A \emph{convex sublattice} of a lattice $L$ is a sublattice of $L$ that is also a convex subset of $L$. If $L$ is a lattice and if $X$ is a subset of $L$, then the \emph{sublattice of $L$ generated by $X$} is the smallest sublattice of $L$ that contains $X$. A lattice $L$ is \emph{finitely generated} if there exists a finite subset $X$ of $L$ such that the sublattice of $L$ generated by $X$ is $L$. If $L$ is a lattice, then an element $a \in L$ is \emph{doubly reducible} if there exist elements $a_1, a_2, a_3, a_4 \in L$ such that $a_1 \parallel a_2$, $a_3 \parallel a_4$, and $a = a_1 \vee a_2 = a_3 \wedge a_4$. \\

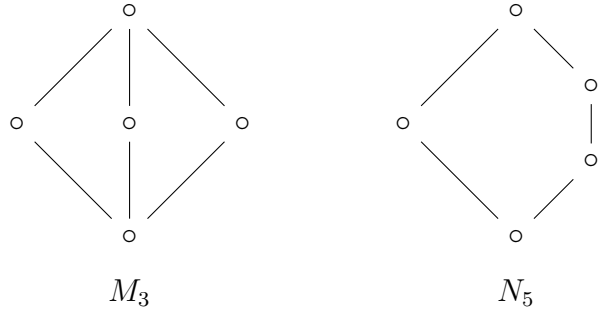
\begin{figure}
\begin{center}
\begin{tikzpicture}[scale=1.5]
\node (x1) at (0,0) {$\circ$};
\node (x2) at (-1,-1) {$\circ$};
\node (x3) at (0,-1) {$\circ$};
\node (x4) at (1,-1) {$\circ$};
\node (x5) at (0,-2) {$\circ$};
\node (x) at (0,-2.5) {$M_3$};

\draw (x1) -- (x2) -- (x5) -- (x3) -- (x1) -- (x4) -- (x5);
\end{tikzpicture} \;\;\;\;\;\;\;\;\;\;\;\;
\begin{tikzpicture}[scale=0.5]
\node (x1) at (0,0) {$\circ$};
\node (x2) at (-3,-3) {$\circ$};
\node (x3) at (2,-2) {$\circ$};
\node (x4) at (2,-4) {$\circ$};
\node (x5) at (0,-6) {$\circ$};
\node (x) at (0,-7.5) {$N_5$};

\draw (x1) -- (x2) -- (x5) -- (x4) -- (x3) -- (x1);
\end{tikzpicture}
\end{center}
\caption{The two lattices in the $M_3$-$N_5$ Theorem.}
\end{figure}
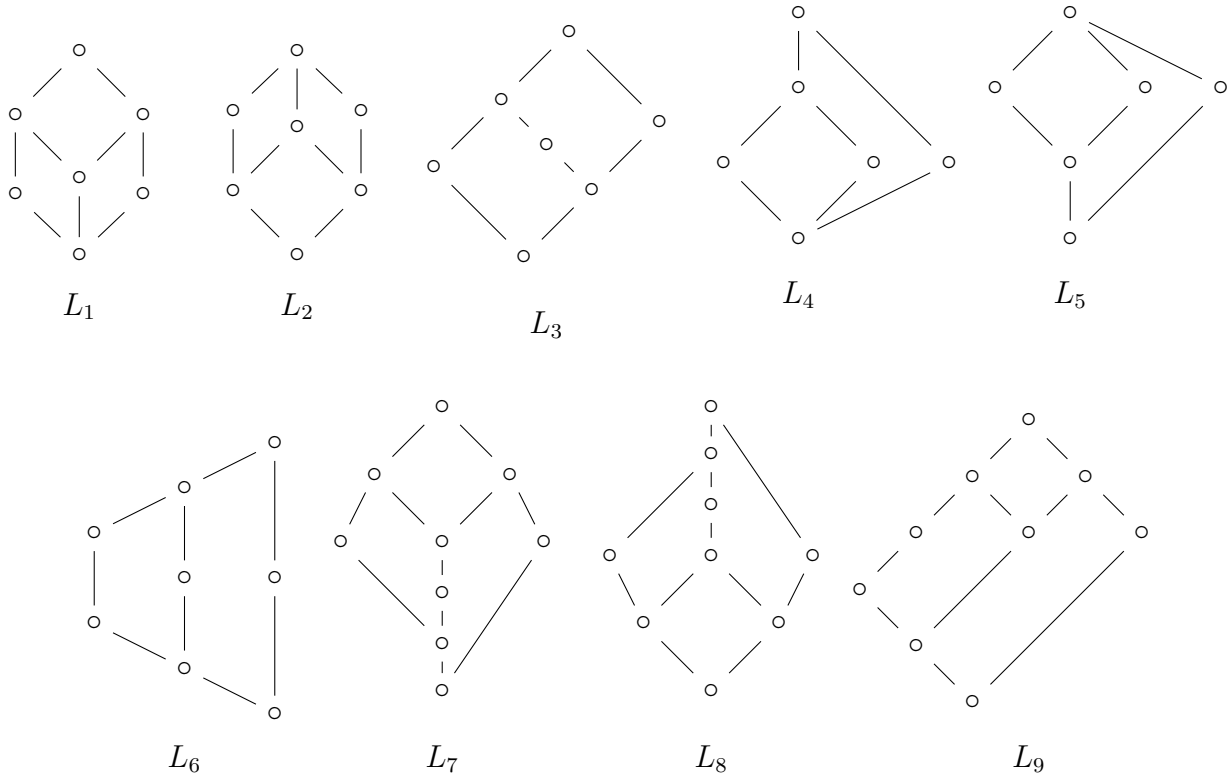
\begin{figure}\label{McKenzies List}

\begin{center}
\begin{tikzpicture}[scale=0.85]
\node (a) at (0,0) {$\circ$};
\node (b) at (-1,-1) {$\circ$};
\node (c) at (1,-1) {$\circ$};
\node (d) at (0,-2) {$\circ$};
\node (e) at (-1,-2.25) {$\circ$};
\node (f) at (1,-2.25) {$\circ$};
\node (g) at (0,-3.2) {$\circ$};
\node (x) at (0,-4) {$L_1$};
\node (y) at (0,-5) {$ $};

\draw (a) -- (c) -- (f) -- (g) -- (e) -- (b) -- (a);
\draw (b) -- (d) -- (c);
\draw (d) --(g);
\end{tikzpicture} \;\;\;\;
\begin{tikzpicture}[scale=0.85]
\node (a) at (0,-3.2) {$\circ$};
\node (b) at (-1,-2.2) {$\circ$};
\node (c) at (1,-2.2) {$\circ$};
\node (d) at (0,-1.2) {$\circ$};
\node (e) at (-1,-0.95) {$\circ$};
\node (f) at (1,-0.95) {$\circ$};
\node (g) at (0,0) {$\circ$};
\node (x) at (0,-4) {$L_2$};
\node (y) at (0,-5) {$ $};

\draw (a) -- (c) -- (f) -- (g) -- (e) -- (b) -- (a);
\draw (b) -- (d) -- (c);
\draw (d) --(g);
\end{tikzpicture} \;\;
\begin{tikzpicture}[scale=0.6]
\node (a) at (0.5,0.5) {$\circ$};
\node (b) at (-1,-1) {$\circ$};
\node (c) at (0,-2) {$\circ$};
\node (d) at (1,-3) {$\circ$};
\node (e) at (2.5,-1.5) {$\circ$};
\node (f) at (-2.5,-2.5) {$\circ$};
\node (g) at (-0.5,-4.5) {$\circ$};
\node (x) at (0,-6) {$L_3$};
\node (y) at (0,-7) {$ $};

\draw (a) -- (e) -- (d) -- (g) -- (f) -- (b) -- (a);
\draw (b) -- (c) -- (d);
\end{tikzpicture} \;
\begin{tikzpicture}
\node (a) at (0,0) {$\circ$};
\node (b) at (0,-1) {$\circ$};
\node (c) at (-1,-2) {$\circ$};
\node (d) at (1,-2) {$\circ$};
\node (e) at (2,-2) {$\circ$};
\node (f) at (0,-3) {$\circ$};
\node (x) at (0,-3.75) {$L_4$};
\node (y) at (0,-4.75) {$ $};

\draw (a) -- (b) -- (c) -- (f) -- (d) -- (b);
\draw (a) -- (e) -- (f);
\end{tikzpicture}
\begin{tikzpicture} 
\node (a) at (0,0) {$\circ$};
\node (b) at (0,1) {$\circ$};
\node (c) at (-1,2) {$\circ$};
\node (d) at (1,2) {$\circ$};
\node (e) at (2,2) {$\circ$};
\node (f) at (0,3) {$\circ$};
\node (x) at (0,-0.75) {$L_5$};
\node (y) at (0,-1.75) {$ $};

\draw (a) -- (b) -- (c) -- (f) -- (d) -- (b);
\draw (a) -- (e) -- (f);
\end{tikzpicture}
\end{center}

\begin{center}
\begin{tikzpicture}[scale=0.6]
\node (a) at (0,0) {$\circ$};
\node (b) at (-2,-1) {$\circ$};
\node (c) at (-4,-2) {$\circ$};
\node (d) at (-2,-3) {$\circ$};
\node (e) at (-4,-4) {$\circ$};
\node (f) at (-2,-5) {$\circ$};
\node (g) at (0,-6) {$\circ$};
\node (x) at (-2,-7) {$L_6$};
\node (h) at (0,-3) {$\circ$};
\draw (a) -- (b) -- (c) -- (e) -- (f) -- (g) -- (h) -- (a);
\draw (b) -- (d) -- (f);
\end{tikzpicture} \;
\begin{tikzpicture}[scale=0.9]
\node (a) at (0,0) {$\circ$};
\node (b) at (-1,-1) {$\circ$};
\node (c) at (1,-1) {$\circ$};
\node (d) at (-1.5,-2) {$\circ$};
\node (e) at (0,-2) {$\circ$};
\node (f) at (1.5,-2) {$\circ$};
\node (g) at (0,-2.75) {$\circ$};
\node (h) at (0,-3.5) {$\circ$};
\node (i) at (0,-4.2) {$\circ$};
\node (x) at (0,-5.2) {$L_7$};

\draw (b) -- (e) -- (c) -- (f) -- (i) -- (h) -- (g) -- (e);
\draw (h) -- (d) -- (b) -- (a) -- (c);
\end{tikzpicture} \;
\begin{tikzpicture}[scale=0.9]
\node (a) at (0,0) {$\circ$};
\node (b) at (-1,1) {$\circ$};
\node (c) at (1,1) {$\circ$};
\node (d) at (-1.5,2) {$\circ$};
\node (e) at (0,2) {$\circ$};
\node (f) at (1.5,2) {$\circ$};
\node (g) at (0,2.75) {$\circ$};
\node (h) at (0,3.5) {$\circ$};
\node (i) at (0,4.2) {$\circ$};
\node (x) at (0,-1) {$L_8$};

\draw (b) -- (e) -- (c) -- (f) -- (i) -- (h) -- (g) -- (e);
\draw (h) -- (d) -- (b) -- (a) -- (c);
\end{tikzpicture}
\begin{tikzpicture}[scale=0.75]
\node (a) at (0,0) {$\circ$};
\node (b) at (-1,-1) {$\circ$};
\node (c) at (1,-1) {$\circ$};
\node (d) at (-2,-2) {$\circ$};
\node (e) at (0,-2) {$\circ$};
\node (f) at (-3,-3) {$\circ$};
\node (g) at (-2,-4) {$\circ$};
\node (h) at (2,-2) {$\circ$};
\node (i) at (-1,-5) {$\circ$};
\node (x) at (0,-6) {$L_9$};

\draw (a) -- (b) -- (d) -- (f) -- (g) -- (i) -- (h) -- (c) -- (a);
\draw (b) -- (e);
\draw (c) -- (e) -- (g);
\end{tikzpicture}
\end{center}

\caption{The first nine lattices in McKenzie's list of subdirectly irreducible lattices (\cite{VL}, p. 19).}
\end{figure}
\begin{figure}

\begin{center}
\begin{tikzpicture}[scale=0.75]
\node (a) at (0,0) {$\circ$};
\node (b) at (1,1) {$\circ$};
\node (c) at (-1,1) {$\circ$};
\node (d) at (2,2) {$\circ$};
\node (e) at (0,2) {$\circ$};
\node (f) at (3,3) {$\circ$};
\node (g) at (2,4) {$\circ$};
\node (h) at (-2,2) {$\circ$};
\node (i) at (1,5) {$\circ$};
\node (x) at (0,-1) {$L_{10}$};
\node (y) at (0,-2) {$ $};

\draw (a) -- (b) -- (d) -- (f) -- (g) -- (i) -- (h) -- (c) -- (a);
\draw (b) -- (e);
\draw (c) -- (e) -- (g);
\end{tikzpicture} \;\;
\begin{tikzpicture}[scale=0.65]
\node (a) at (0,0) {$\circ$};
\node (b) at (-1,-1) {$\circ$};
\node (c) at (1,-1) {$\circ$};
\node (d) at (-2,-2) {$\circ$};
\node (e) at (0,-2) {$\circ$};
\node (f) at (-3,-3) {$\circ$};
\node (g) at (-1,-3) {$\circ$};
\node (h) at (3,-3) {$\circ$};
\node (i) at (1,-5) {$\circ$};
\node (j) at (0,-6) {$\circ$};
\node (x) at (0,-7) {$L_{11}$};
\node (y) at (0,-8) {$ $};
\draw (d) -- (f) -- (j) -- (i) -- (g) -- (d) -- (b) -- (a) -- (c) -- (h) -- (i);
\draw (b) -- (e);
\draw (g) -- (e) -- (c);
\end{tikzpicture} \;\;
\begin{tikzpicture}[scale=0.65]
\node (a) at (0,0) {$\circ$};
\node (b) at (-1,1) {$\circ$};
\node (c) at (1,1) {$\circ$};
\node (d) at (-2,2) {$\circ$};
\node (e) at (0,2) {$\circ$};
\node (f) at (-3,3) {$\circ$};
\node (g) at (-1,3) {$\circ$};
\node (h) at (3,3) {$\circ$};
\node (i) at (1,5) {$\circ$};
\node (j) at (0,6) {$\circ$};
\node (x) at (0,-1) {$L_{12}$};
\node (y) at (0,-2) {$ $};

\draw (d) -- (f) -- (j) -- (i) -- (g) -- (d) -- (b) -- (a) -- (c) -- (h) -- (i);
\draw (b) -- (e);
\draw (g) -- (e) -- (c);
\end{tikzpicture}
\end{center}
\begin{center}
\begin{tikzpicture}[scale=1.3]
\node (a) at (0,0) {$\circ$};
\node (b) at (-1.2,-1) {$\circ$};
\node (c) at (0,-1) {$\circ$};
\node (d) at (1.2,-1) {$\circ$};
\node (e) at (-1.2,-2) {$\circ$};
\node (f) at (0,-2) {$\circ$};
\node (g) at (1.2,-2) {$\circ$};
\node (h) at (0,-3) {$\circ$};
\node (hh) at (-0.6,-2.5) {$\circ$};
\node (x) at (0,-4) {$L_{13}$};

\draw (a) -- (d) -- (g) -- (h) -- (hh) -- (e) -- (b) -- (a);
\draw (b) -- (f) -- (d);
\draw (e) -- (c) -- (g);
\draw (a) -- (c);
\draw (f) -- (h);
\end{tikzpicture} \;\;\;\;\;\;\;\;
\begin{tikzpicture}[scale=1.3]
\node (a) at (0,0) {$\circ$};
\node (b) at (-1.2,-1) {$\circ$};
\node (c) at (0,-1) {$\circ$};
\node (d) at (1.2,-1) {$\circ$};
\node (e) at (-1.2,-2) {$\circ$};
\node (f) at (0,-2) {$\circ$};
\node (g) at (1.2,-2) {$\circ$};
\node (h) at (0,-3) {$\circ$};
\node (aa) at (0.6,-0.5) {$\circ$};
\node (x) at (0,-4) {$L_{14}$};
\draw (a) -- (aa) -- (d) -- (g) -- (h) -- (e) -- (b) -- (a);
\draw (b) -- (f) -- (d);
\draw (e) -- (c) -- (g);
\draw (a) -- (c);
\draw (f) -- (h);
\end{tikzpicture} \;\;\;\;\;\;\;\;
\begin{tikzpicture}[scale = 0.7]
\node (a) at (0,0) {$\circ$};
\node (b) at (-1,-1) {$\circ$};
\node (c) at (1,-1) {$\circ$};
\node (d) at (0,-2) {$\circ$};
\node (e) at (-3,-3) {$\circ$};
\node (f) at (0,-4) {$\circ$};
\node (g) at (3,-3) {$\circ$};
\node (h) at (-1,-5) {$\circ$};
\node (i) at (1,-5) {$\circ$};
\node (j) at (0,-6) {$\circ$};
\node (x) at (0,-7) {$L_{15}$};

\draw (a) -- (c) -- (g) -- (i) -- (j) -- (h) -- (e) -- (b) -- (a);
\draw (b) -- (d) -- (c);
\draw (h) -- (f) -- (i);
\draw (d) -- (f);
\end{tikzpicture}
\end{center}

\caption{The last six lattices in McKenzie's list of subdirectly irreducible lattices (\cite{VL}, p. 19)}
\end{figure}
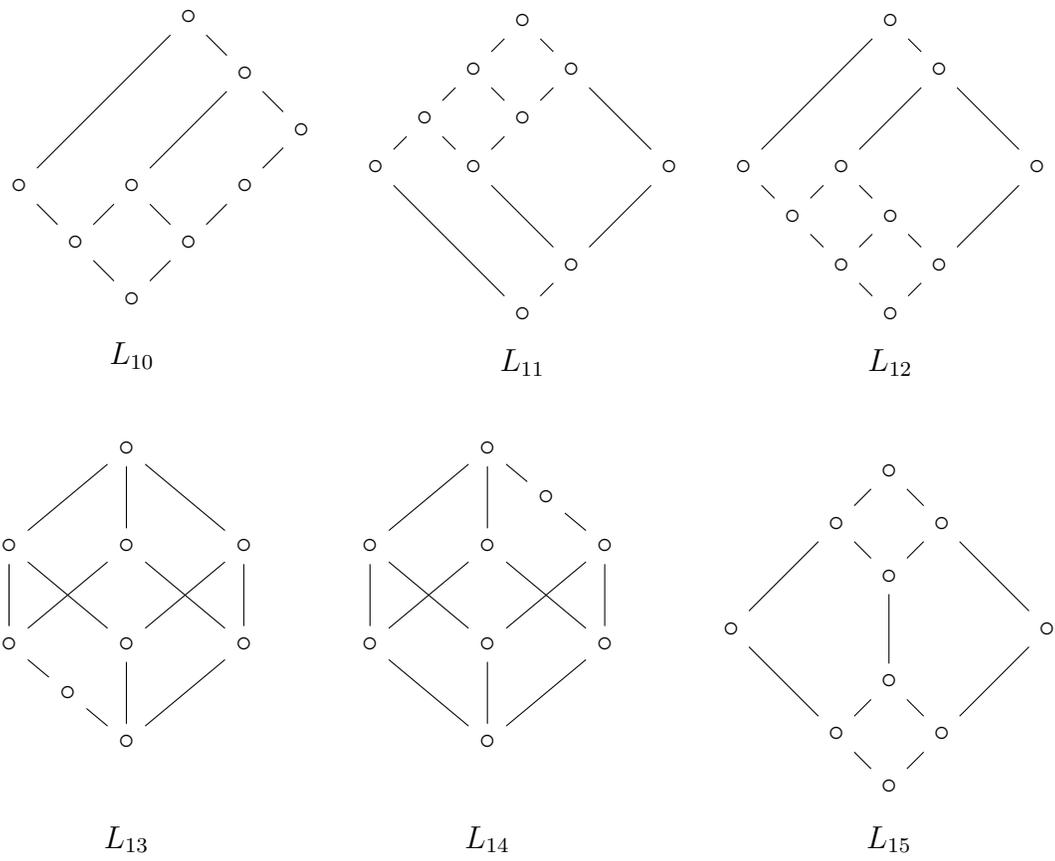

If $K$ and $L$ are lattices, then a \emph{lattice homomorphism $f : K \to L$} is a function from the set of elements of $K$ to the set of elements of $L$ such that for all $a,b \in K$, $f(a \vee b) = f(a) \vee f(b)$ and $f(a \wedge b) = f(a) \wedge f(b)$. If $K$ and $L$ are lattices, then $K$ is \emph{isomorphic to} $L$ if there exist lattice homomorphisms $f : K \to L$ and $g : L \to K$ such that $f$ and $g$ are bijections, $g \circ f$ is the identity map on $K$, and $f \circ g$ is the identity map on $L$. Lastly, call a bijective lattice homomorphism a \emph{lattice isomorphism}. \\

Recall that if $S$ is a set, then a \emph{free lattice on $S$} is a lattice $FL(S)$ that satisfies the following \emph{universal property}. For all lattices $L$ and for all functions $f$ from $S$ to the set of elements of $L$, there exists a unique lattice homomorphism $g : FL(S) \to L$ such that for all $s \in S$, $g(s) = f(s)$ (\cite{FL}, p. 136). Any two free lattices on $S$ are isomorphic, so we say that $FL(S)$ is \emph{the} free lattice on $S$ (\cite{FL}, p. 136). \\

If $P$ and $Q$ are posets, then define their \emph{direct product} $P \times Q$ to be the poset whose set of elements is $\{(p, q) : p \in P \text{ and } q \in Q \}$ and where $(p_1, q_1) \leq (p_2, q_2)$ if and only if $p_1 \leq p_2$ and $q_1 \leq q_2$. Similarly, if $S_1$ and $S_2$ are sets, then define $S_1 \times S_2 = \{(s_1, s_2) : s_1 \in S_1 \text{ and } s_2 \in S_2\}$. A poset $P$ is a \emph{chain} if for all $a,b \in P$, $a \leq b$ or $b \leq a$. With that definition in mind, let, for all $n \in \mathbb{N}$, $\textbf{n}$ denote the $n$-element chain. We will also consider $\mathbb{Z}$ as a chain with partial order defined by $\cdots < -2 < -1 < 0 < 1 < 2 < \cdots$. \\

An important condition in free lattice theory is \emph{Whitman's Condition}, discovered by Whitman in the 40's when finding a solution to the \emph{Word Problem for Free Lattices}.

\begin{definition}(Whitman \cite{Whitman 1, Whitman 2}) A lattice $L$ satisfies \emph{Whitman's Condition} if for all $a,b,c,d \in L$ satisfying $a \wedge b \leq c \vee d$, $a \leq c \vee d$, $b \leq c \vee d$, $a \wedge b \leq c$, or $a \wedge b \leq d$.
\end{definition}

He proved that all sublattices of free lattices satisfy Whitman's condition.

\begin{theorem}(Whitman \cite{Whitman 1, Whitman 2})\label{Whitmans Theorem} All sublattices of free lattices satisfy Whitman's condition.
\end{theorem}

A fact that we will continually use is that if a lattice $L$ satisfies Whitman's Condition, then $L$ has no doubly reducible elements. In particular, by Theorem \ref{Whitmans Theorem}, any sublattice of a free lattice has no doubly reducible elements. Another important property discovered about sublattices of free lattices was discovered at around 1960 by J\'onsson.

\begin{definition}(J\'onsson \cite{Semidistributive origins 1, Semidistributive origins 2}) A lattice $L$ is \emph{semidistributive} if it satisfies the following \emph{semidistributive laws}.
\begin{itemize}
\item For all $a,b,c,d \in L$, if $a \vee b = d$ and $a \vee c = d$, then $a \vee (b \wedge c) = d$.
\item For all $a,b,c,d \in L$, if $a \wedge b = d$ and $a \wedge c = d$, then $a \wedge (b \vee c) = d$.
\end{itemize}
\end{definition}

He proved the following result.

\begin{theorem}(J\'onsson \cite{Semidistributive origins 1, Semidistributive origins 2}) All sublattices of free lattices are semidistributive.
\end{theorem}

We also recall the notion of a distributive lattice.

\begin{definition}\label{distributive} (\cite{ILO}) A lattice $L$ is \emph{distributive} if the following \emph{distributive laws} are satisfied.
\begin{itemize}
\item For all $a,b,c \in L$, $a \vee (b \wedge c) = (a \vee b) \wedge (a \vee c)$.
\item For all $a,b,c \in L$, $a \wedge (b \vee c) = (a \wedge b) \vee (a \wedge c)$.
\end{itemize}
\end{definition}

Distributive sublattices of free lattices have been classified, and are, structurally, very simple.

\begin{theorem}\label{Galvin and Jonsson}(Galvin and J\'onsson \cite{Distributive sublattices}) A distributive lattice $D$ is a distributive sublattice of a free lattice if and only if there exists a set partition $\mathcal{F}$ of the set of elements of $D$ that satisfies all of the following properties.
\begin{itemize}
\item $|\mathcal{F}|$ is countable.
\item For all distinct $A, B \in \mathcal{F}$, $a < b$ in $P$ for all $a \in A$ and $b \in B$, or $b < a$ for all $a \in A$ and $b \in B$. 
\item For all $A \in \mathcal{F}$, $A$ is a countable chain, $A$ is isomorphic to $\textbf{2} \times C$ for some countable chain $C$, or $A$ is isomorphic to the three atom Boolen algebra.
\end{itemize}
\end{theorem}

\begin{remark} Galvin and J\'onsson \cite{Distributive sublattices} proved Theorem \ref{Galvin and Jonsson} for distributive lattices that have no doubly reducible elements. In particular, as stated by Galvin and J\'onsson in \cite{Distributive sublattices}, Theorem \ref{Galvin and Jonsson} is a characterization of distributive lattices with no doubly reducible elements.
\end{remark}

\begin{example}\label{Galvin and Jonsson ex} Consider the lattice $\textbf{2} \times \mathbb{Z}$. By Theorem \ref{Galvin and Jonsson}, $\textbf{2} \times \mathbb{Z}$ is a distributive sublattice of a free lattice. Moreover, it is neither finitely generated nor projective. It is not projective because it does not satisfy a property known as the \emph{minimal join cover refinement property} \cite{FL}.
\end{example}

Another class of lattices that we will refer to in this paper is as follows.

\begin{definition}\label{modular} (\cite{ILO}) A lattice $L$ is \emph{modular} if it satisfies the \emph{modular law}: For all $a, b, c \in L$ such that $a \leq c$, 
$$(a \vee b) \wedge c = a \vee (b \wedge c).$$
\end{definition}

Recall that the left-most lattice in Figure 1 is called the \emph{diamond}, which is denoted by $M_3$, and the right-most lattice in Figure 1 is called the \emph{pentagon}, which is denoted by $N_5$. An important property of modular and distributive lattices is the following result, established by Dedekind and Birkhoff, called the $M_3$-$N_5$ \emph{Theorem}.

\begin{theorem}(The $M_3$-$N_5$ Theorem, Dedekind and Birkhoff \cite{ILO}) A lattice $L$ is modular if and only if $L$ does not have a sublattice that is isomorphic to $N_5$. Moreover, a lattice $L$ is distributive if and only if $L$ does not have a sublattice that is isomorphic to $M_3$ or to $N_5$.
\end{theorem}

By the $M_3$-$N_5$ Theorem, and the fact that $M_3$ is not semidistributive, a sublattice of a free lattice is modular if and only if it is distributive. \\

In this paper, we will use terminology on varieties of lattices from \cite{VL}. In particular, recall the notion of a \emph{variety of lattices}, \emph{subdirectly irreducible lattices}, the \emph{variety generated by a lattice}, and the \emph{lattice of lattice varieties}. Note that the class $\mathcal{D}$ of distributive lattices is a variety of lattices. In fact, recall that $\mathcal{D}$ is the smallest non-trivial variety of lattices. Furthermore, recall the following. A variety $\mathcal{V}$ of lattices is \emph{join-irreducible} if there do not exist varieties $\mathcal{V}_1$ and $\mathcal{V}_2$ of lattices such that $\mathcal{V}_1 \neq \mathcal{V}$, $\mathcal{V}_2 \neq \mathcal{V}$, and $\mathcal{V}$ is the smallest variety that contains $\mathcal{V}_1$ and $\mathcal{V}_2$. Moreover, a variety $
\mathcal{V}_2$ of lattices \emph{covers} a variety $\mathcal{V}_1$ of lattices if $\mathcal{V}_1 \neq \mathcal{V}_2$, $\mathcal{V}_1 \subset \mathcal{V}_2$, and, for all varieties $\mathcal{V}$ of lattices, $\mathcal{V}_1 \subseteq \mathcal{V} \subseteq \mathcal{V}_2$ implies that $\mathcal{V} = \mathcal{V}_1$ or $\mathcal{V} = \mathcal{V}_2$. Lastly, if $\mathcal{V}$ is a variety of lattices, then write $L \in \mathcal{V}$ to mean that $L$ is a member of $\mathcal{V}$ and $L \notin \mathcal{V}$ to mean otherwise. For example, writing $L \in \mathcal{D}$ is equivalent to saying that $L$ is a distributive lattice. \\

Let $\mathcal{N}$ denote the variety of lattices that is generated by $N_5$. This variety of lattices is the smallest variety of lattices that contains the variety $\mathcal{D}$ of distributive lattices and that contains non-distributive sublattices of free lattices. To compare, it is known \cite{VL} that the variety of distributive lattice is the variety of lattices that is generated by the two element chain.

\begin{example}\label{Ninf} A simple example of a sublattice $L$ of a free lattice, where $L$ is not distributive and where $L \in \mathcal{N}$, is the lattice depicted below.

\begin{center}
\begin{tikzpicture}[scale=1]
\node (x1) at (0,0) {$\circ$};
\node (x2) at (-3,-3) {$\circ$};
\node (x3) at (1.8,-2) {$\circ$};
\node (x33) at (2,-2.66) {$\circ$};
\node (x44) at (2,-3.33) {$\circ$};
\node (x4) at (1.8,-4) {$\circ$};
\node (x5) at (0,-6) {$\circ$};

\node (x555) at (1.2,-4.66) {$\circ$};

\node (x55) at (0.6,-5.33) {$ $};

\node (x111) at (1.2,-1.33) {$\circ$};

\node (x11) at (0.6,-0.66) {$ $};

\node (1x1) at (0.3,-0.33) {$.$};
\node (2x1) at (0.45,-0.495) {$.$};
\node (3x1) at (0.57,-0.62) {$.$};

\node (1x5) at (0.3,-5.67) {$.$};
\node (2x5) at (0.45,-5.505) {$.$};
\node (3x5) at (0.57,-5.38) {$.$};

\draw (x11) -- (x111) -- (x3) -- (x33) -- (x44) -- (x4) -- (x555) -- (x55);
\draw (x5) -- (x2) -- (x1);

\end{tikzpicture}
\end{center}

One can check that for any finitely generated sublattice $L'$ of $L$, there exists a positive integer $k$ such that $L'$ is a sublattice of the $k$-fold direct product $N_5 \times N_5 \times \cdots \times N_5$. Hence, $L \in \mathcal{N}$. Moreover, as $FL(\omega)$ is a dense partial order and $N_5$ is a sublattice of $FL(\omega)$, $L$ is a sublattice of the free lattice $FL(\omega)$.
\end{example}

\begin{remark} As in Example \ref{Galvin and Jonsson ex}, the lattice in Example \ref{Ninf} is neither finitely generated nor projective. It is not projective because that lattice does not satisfy the minimal join cover refinement property.
\end{remark}

In 1972, McKenzie gave a list of fifteen subdirectly irreducible lattices $L_i$, $1 \leq i \leq 15$ with the property that for all $1 \leq i \leq 15$, the variety $\mathcal{L}_i$ of lattices generated by $L_i$ is a join-irreducible cover of $\mathcal{N}$ \cite{VL, EBNMV}. The list is given in Figure 2 and Figure 3. Later on, in 1979, J\'onsson and Rival proved that there are no other varieties of lattices that cover $\mathcal{N}$ \cite{JonssonRival, VL}. \\

As all sublattices of free lattices are semidistributive, the following definition will be very useful to us.

\begin{definition} (\cite{VL}, p. 81) A variety $\mathcal{V}$ of lattices is \emph{semidistributive} if every member of $\mathcal{V}$ is semidistributive.
\end{definition}

Semidistributive varieties can be characterized as follows.

\begin{theorem}(J\'onsson and Rival (\cite{VL}, p. 81), \cite{JonssonRival})\label{Jonsson Rival} A variety $\mathcal{V}$ of lattices is semidistributive if and only if $M_3, L_1, L_2, L_3, L_4, L_5 \notin \mathcal{V}$.
\end{theorem}

In particular, Theorem \ref{Jonsson Rival} implies that $\mathcal{N}$ is a semidistributive variety of lattices. In this paper, we will make essential use of all of all of the semidistributive lattices in McKenzie's list, which are the lattices $L_i$ for $6 \leq i \leq 15$. Because $L_i \notin \mathcal{L}_j$ for all satisfying $i \neq j$, it follows, by Theorem \ref{Jonsson Rival}, that for all $6 \leq i \leq 15$, $\mathcal{L}_i$ is a semidistributive variety.

\section{A lemma involving the lattice $L_{15}$}\label{sec:L15}

Before proving the three structural results of this paper, we prove a technical lemma that involves the lattice $L_{15}$.

\begin{lemma}\label{L15} Let $L$ be a lattice that has no doubly reducible elements, and let $a_1$, $a_2$, $a_3$, $b_1$, $b_2$, and $b_3$ be six distinct elements in $L$ such that $a_1 < a_2 < a_3$, $b_1 < b_2 < b_3$, $a_1 < b_3$, and $b_1 < a_3$. Moreover, assume that $a_2 \parallel b_i$ for all $1 \leq i \leq 3$, that $a_i \parallel b_2$ for all $1 \leq i \leq 3$, that $a_2 \vee b_2 = a_3 \vee b_3$, and that $a_2 \wedge b_2 = a_1 \wedge b_1$. Then $L$ has a sublattice that is isomorphic to $L_{15}$.
\end{lemma}

\begin{proof} If $a_1 \leq b_1$, then $a_1 \leq b_1 \leq b_2$, contradicting the assumption that $a_1 \parallel b_2$. So $a_1 \nleq b_1$. By symmetry, $b_1 \nleq a_1$, $a_3 \ngeq b_3$, and $b_3 \ngeq a_3$. Hence, $a_1 \parallel b_1$ and $a_2 \parallel b_2$. Let $a_1' = a_2 \wedge b_3$ and let $b_1' = a_3 \wedge b_2$. Since $a_2 \parallel b_3$, it follows that $a_1' < a_2$ and since $a_3 \parallel b_2$, it follows that $b_1' < b_2$. Moreover, $a_1' \parallel b_2$ because $a_1' \leq b_2$ implies that
$$a_1 \leq a_1' = a_1' \wedge b_2 \leq a_2 \wedge b_2 = a_1 \wedge b_1 < a_1$$
which is impossible, and $a_1' \geq b_2$ implies that $a_2 \geq a_1' \geq b_2$, contrary to the assumption that $a_2 \parallel b_2$. By symmetry, $a_2 \parallel b_1'$. If $a_1' \geq b_1'$, then $b_1' \leq a_1' \leq a_2$. But then, $b_1' \leq a_2'$, which is impossible. So $a_1' \ngeq b_1'$. By symmetry, $a_1' \nleq b_1'$. Hence, $a_1' \parallel b_1'$. Lastly, as $a_1 \leq a_1' \leq a_2$ and as $b_1 \leq b_1' \leq b_2$, $a_2 \wedge b_2 = a_1' \wedge b_1'$. Next, let $a_3'' = a_2 \vee b_1'$ and let $b_3'' = a_1' \vee b_2$. The dual of the above argument implies that $a_2 < a_3''$, $b_2 < b_3''$, $a_2 \parallel b_3''$, $a_3'' \parallel b_2$, $a_3'' \parallel b_3''$, and $a_2 \vee b_2 = a_3'' \vee b_3''$. Since $a_1' \leq b_3'' \leq b_3$ and $a_1' = a_2 \wedge b_3$, $a_1' = a_2 \wedge b_3''$. Moreover, since $b_1' \leq a_3'' \leq a_3$ and $b_1' = a_3 \wedge b_2$, $b_1' = a_3'' \wedge b_2$. By symmetry, $a_3'' = a_2 \vee b_1'$ and $b_3'' = a_1' \vee b_2$. \\

The lattice $L$ has no doubly reducible elements, so, as $a_1' \parallel b_1'$ and $a_3'' \parallel b_3''$, $a_1' \vee b_1' < a_3'' \wedge b_3''$. If $a_2 \geq a_1' \vee b_1'$, then $a_2 \geq a_1' \vee b_1' > b_1'$, which is impossible, and if $a_2 \leq a_1' \vee b_1'$, then $a_2 \leq a_1' \vee b_1' < a_3'' \wedge b_3'' < b_3''$, which is also impossible. Hence, $a_2 \parallel a_1' \vee b_1'$. Similarly, $a_2 \parallel a_3'' \wedge b_3''$. By symmetry, it follows that $a_1' \vee b_1' \parallel b_2$ and that $a_3'' \wedge b_3'' \parallel b_2$. Furthermore, as $a_1' < a_1' \vee b_1' < a_3'' \wedge b_3'' < b_3''$ and $a_1' = a_2 \wedge b_3''$, $a_1' = a_2 \wedge (a_1' \vee b_1') = a_2 \wedge (a_3'' \wedge b_3'')$. By symmetry, and the facts that $a_3'' = a_2 \vee b_1'$, $b_1' = a_3'' \wedge b_2$, and $b_3'' = a_1' \vee b_2$, it follows that $a_3'' = a_2 \vee (a_1' \vee b_1') = a_2 \vee (a_3'' \wedge b_3'')$, $b_1' = b_2 \wedge (a_1' \vee b_1') = b_2 \wedge (a_3'' \wedge b_3'')$, and $b_3'' = b_2 \vee (a_1' \vee b_1') = b_2 \vee (a_3'' \wedge b_3'')$. Hence, $L$ contains the sublattice depicted below, and it is isomorphic to $L_{15}$.

\begin{center}
\begin{tikzpicture}
\node (a2b2t) at (0,0) {$a_2 \vee b_2$};
\node (a3) at (-1.5,-1) {$a_3''$};
\node (b3) at (1.5,-1) {$b_3''$};
\node (a3b3) at (0,-2) {$a_3'' \wedge b_3''$};

\node (a1b1) at (0,-3) {$a_1' \vee b_1'$};
\node (a1) at (-1.5,-4) {$a_1'$};
\node (b1) at (1.5,-4) {$b_1'$};
\node (a2b2b) at (0,-5) {$a_2 \wedge b_2$};

\node (a2) at (-2.5,-2.5) {$a_2$};
\node (b2) at (2.5,-2.5) {$b_2$};

\draw (a2b2b) -- (a1) -- (a2) -- (a3) -- (a2b2t) -- (b3) -- (b2) -- (b1) -- (a2b2b);
\draw (a3) -- (a3b3) -- (a1b1) -- (b1);
\draw (a1) -- (a1b1);
\draw (a3b3) -- (b3);
\end{tikzpicture}
\end{center}

\end{proof}

\section{Atoms and coatoms}\label{sec:atoms and coatoms}

A consequence of Galvin and J\'onsson's classification of distributive sublattices of free lattices is that in such a lattice, every antichain that has three elements satisfies some very strong conditions. In this section, we prove the first main result of this paper by proving a structural property of sublattices of a free lattice that are in the variety generated by $N_5$ by proving that the sublattices of free lattices that we are interested in satisfy a property similar to the above conditions. \\

We now state and prove the first structural property of this paper.

\begin{theorem}\label{cube theorem} Let $L \in \mathcal{N}$, and assume that $L$ is isomorphic to a sublattice of a free lattice. Furthermore, assume that $Y$ is an antichain in $L$. If there is an element $d \in L$ such that $a \wedge b = d$ for all distinct elements $a, b \in Y$, then $|Y| \leq 3$ and at most two elements of $Y$ do not cover $d$ in $L$. Moreover, if there is an element $d \in L$ such that $a \vee b = d$ for all distinct elements $a, b \in Y$, then $|Y| \leq 3$ and at most two elements of $Y$ are not covered by $d$ in $L$.
\end{theorem}

A stronger form of Theorem \ref{cube theorem} is satisfied by distributive sublattices of free lattices.

\begin{example} By Theorem \ref{Galvin and Jonsson}, we have the following. Let $L$ be a distributive sublattice of a free lattice, and let $Y$ be an antichain in $L$. Then, $|Y| \leq 3$. Moreover, if there is an element $d \in L$ such that $a \wedge b = d$ for all distinct elements $a, b \in Y$, then $|Y| = 2$ implies that at most one element of $Y$ does not cover $d$ in $L$, and $|Y| = 3$ implies that every element of $Y$ covers $d$ in $L$. Furthermore, if there is an element $d \in L$ such that $a \vee b = d$ for all distinct elements $a, b \in Y$, then $|Y| = 2$ implies that at most one element of $Y$ is not covered by $d$ in $L$, and $|Y| = 3$ implies that every element of $Y$ is covered by $d$ in $L$.
\end{example}

For an additional comparison, we note the following.

\begin{remark} A weaker form of Theorem \ref{cube theorem} holds for all sublattices of free lattices. Baldwin, Berman, Glass, and Hodges  \cite{Baldwin et al} used a special case of Erd\H{o}s and Rado's $\Delta$-system Lemma \cite{Erdos and Rado} to prove that if $L$ is a sublattice of a free lattice and if $Y$ is an antichain in $L$ such that for some $d \in L$, $a \wedge b = d$ for all distinct $a, b \in Y$, then $Y$ is countable.
\end{remark}

We now prove Theorem \ref{cube theorem}.

\begin{proof} Firstly, assume that $Y = \{a, b, c\}$ is an antichain in $L$ for some distinct elements $a,b,c \in L$. If $a \vee b = c \vee a$ and $a \vee b \neq b \vee c$, then, as $b \vee c \leq a \vee b \vee c = (a \vee b) \vee (c \vee a) = a \vee b$, $a \vee b > b \vee c$. So, if $a \vee b = b \vee c$, $b \vee c = c \vee a$, or $c \vee a = a \vee b$, then assume without loss of generality that $a \vee b \geq b \vee c$ and $a \vee b \geq c \vee a$. Since $a \vee b \geq b \vee c$, $c \wedge (a \vee b) \geq c \wedge (b \vee c) = c$. So, $c \wedge (a \vee b) = c$. But as $L$ is semidistributive and as $c \wedge a = c \wedge b = d$, that is impossible. Hence, assume without loss of generality that $a \vee b \neq b \vee c$, $b \vee c \neq c \vee a$, and $c \vee a \neq a \vee b$. Then
$$(a \vee b) \vee (c \vee a) = (a \vee b) \vee (b \vee c) = (c \vee a) \vee (b \vee c) = a \vee b \vee c.$$

Let $a_* = (a \vee b) \wedge (c \vee a)$, let $b_* = (a \vee b) \wedge (b \vee c)$, and let $c_* = (c \vee a) \wedge (b \vee c)$. Since $a_* \geq a$, $b_* \geq b$, and $c_* \geq c$, $a_* \vee b_* = a \vee b$, $b_* \vee c_* = b \vee c$, $a_* \vee c_* = a \vee c$, and $a_* \vee b_* \vee c_* = a \vee b \vee c$. Since $L$ is semidistributive, the fact that $c \wedge a = c \wedge b = d$ implies that $c \wedge (a \vee b) = d$. So, as $a_* \leq a \vee b$, $d = a \wedge c \leq a_* \wedge c \leq (a \vee b) \wedge c = d$ and it follows that $a_* \wedge c = d$. By symmetry, $a_* \wedge b = d$, so as $L$ is semidistributive, $a_* \wedge (b \vee c) = d$. Hence, as $b_* \leq b \vee c$, it follows that $a_* \wedge b_* = d$. By symmetry, $b_* \wedge c_* = d$ and $c_* \wedge a_* = d$. Since $L$ is semidistributive, since $a_* \wedge b = d$, and since $a_* \wedge c = d$, $a_* \wedge (b \vee c) = d$. By symmetry, $b_* \wedge (c \vee a) = d$ and $c_* \wedge (a \vee b) = d$. Lastly, $a_* \vee b_* = a \vee b$, $b_* \vee c_* = b \vee c$, and $c_* \vee a_* = c \vee a$. Hence, the sublattice of $L$ generated by $\{a_*, b_*, c_*, a \vee b, b \vee c, c \vee a, d\}$ is isomorphic to the three atom Boolean algebra $\textbf{2} \times \textbf{2} \times \textbf{2}$. \\

Suppose for a contradiction that $a$ does not cover $d$ in $L$, that $b$ does not cover $d$ in $L$, and that $c$ does not cover $d$ in $L$. Then there exist elements $a', b', c' \in L$ such that $d < a' < a$, $d < b' < b$, and $d < c' < c$. In particular, $d < a' < a_*$, $d < b' < b_*$, and $d < c' < c_*$. If $a' \leq b_*$, then $a' \leq b_* \wedge a_* = d$, which is impossible. Moreover, if $a' \geq b_*$, then $b_* \leq a' < a_*$, which is also impossible. So $a' \parallel b_*$. By symmetry, $a' \parallel c_*$, $b' \parallel a_*$, $b' \parallel c_*$, $c' \parallel a_*$, and $c' \parallel b_*$. As $L$ is semidistributive, the fact that $a' \wedge b_* = a' \wedge c_* = d$ implies that $a' \wedge (b_* \wedge c_*) = d$. So $a' \parallel b_* \vee c_*$. By symmetry, $b' \parallel c_* \vee a_*$ and $c' \parallel a_* \vee b_*$. \\

If $a' \vee b_* = a_* \vee b_*$ and $a' \vee c_* = a_* \vee c_*$, then $a' \vee (b_* \vee c_*) = a_* \vee b_* \vee c_*$, implying, as $a' \leq a_*$, that the sublattice of $L$ generated by $\{a_*, b_*, c_*, d, a'\}$ is isomorphic to $L_{13}$. But that is impossible because $L_{13} \notin \mathcal{N}$. Hence, $a' \vee b_* < a_* \vee b_*$ or $a' \vee c_* < a_* \vee c_*$. By symmetry, it follows that $b' \vee c_* < b_* \vee c_*$ or $b' \vee a_* < b_* \vee a_*$, and it follows that $c' \vee a_* < c_* \vee a_*$ or $c' \vee b_* < c_* \vee b_*$. Therefore, it is enough to consider the following. \\

If $a' \vee b_* < a_* \vee b_*$ and $b' \vee a_* < a_* \vee b_*$, then consider the following. As $a_* \parallel b'$ and $a' \parallel b_*$, the inequalities $a' < a_* < a_* \vee b'$, $b' < b_* < a' \vee b_*$, $a' < a' \vee b_*$, and $b' < a_* \vee b'$ hold. Moreover, $a_* \vee b_* = (a_* \vee b') \vee (a' \vee b_*)$ and $a_* \wedge b_* = a' \wedge b'$. So it is enough to note the following. If $a_* \geq b'$, then $b' \leq a_* \wedge b_* = d$, but that is impossible. So $a_* \ngeq b'$. Moreover, if $a_* \leq b'$, then $a_* \leq b' < b_*$, but that is impossible as $a_* \parallel b_*$. So $a_* \nleq b'$. Hence, $a_* \parallel b'$. If $a_* \leq a' \vee b_*$, then $a_* \vee b_* \leq (a' \vee b_*) \vee b_* = a' \vee b_* \leq a_* \vee b_*$, implying that $a' \vee b_* = a_* \vee b_*$. But that is contrary to the fact that $a' \vee b_* < a_* \vee b_*$. So $a_* \nleq a' \vee b_*$. If $a_* \geq a' \vee b_*$, then $a_* \geq a' \vee b_* \geq b_*$, which is impossible. So $a_* \ngeq a' \vee b_*$. Hence, $a_* \parallel a' \vee b_*$. By symmetry, $a' \parallel b_*$ and $a_* \vee b' \parallel b_*$. Hence, as $L$ has no doubly reducible elements, $a'$, $a_*$, $a_* \vee b'$, $b'$, $b_*$, and $a' \vee b_*$ satisfy the hypothesis of Lemma \ref{L15}. This is depicted below, recall that $a_* \vee b_* = a \vee b$ and that $a_* \wedge b_* = d$.

\begin{center}
\begin{tikzpicture}[scale=1.5]
\node (B) at (-1.5,0) {$a \vee b$};
\node (E) at (-1.5,-2) {$a_*$};
\node (F) at (0.5,-1) {$b_*$};
\node (H) at (0.5,-3) {$d$};
\node (B2) at (-1.5,-1) {$a_* \vee b'$};
\node (H2) at (0.5,-2) {$b'$};
\node (B1) at (-0.5,-0.5) {$a' \vee b_*$};
\node (H1) at (-0.5,-2.5) {$a'$};

\draw (B1) -- (H1);
\draw (B2) -- (H2);
\draw (B) -- (B2) -- (E) -- (H1) -- (H) -- (H2) -- (F) -- (B1) -- (B);
\end{tikzpicture}
\end{center}
 
So by Lemma \ref{L15}, $L$ has a sublattice that is isomorphic to $L_{15}$. But that is impossible because $L_{15} \notin \mathcal{N}$. So it is impossible for $a' \vee b_* < a_* \vee b_*$ and $b' \vee a_* < a_* \vee b_*$. \\

If $a' \vee b_* < a_* \vee b_*$ and $b' \vee c_* < b_* \vee c_*$, then $a' \vee b' = a' \vee b_*$, $(a' \vee b_*) \vee (b \vee c) < a \vee b \vee c$, or $a' \vee b' < a' \vee b_*$ and $(a' \vee b_*) \vee (b \vee c) = a \vee b \vee c$. \\

If $a' \vee b' = a' \vee b_*$, then $a'$, $b'$, $a \vee b$, and $b \vee c$ violate Whitman's Condition. This is because, as indicated in the following diagram, $(a \vee b) \wedge (b \vee c) = b_* < a' \vee b'$, $(a \vee b) \wedge (b \vee c) \nleq a'$, $(a \vee b) \wedge (b \vee c) \nleq b'$, $a' \vee b' \ngeq a \vee b$, and, as $a' \vee b' \geq b \vee c$ would imply $a \vee b \geq a' \vee b' \geq b \vee c$, $a' \vee b' \ngeq b \vee c$.

\begin{center}
\begin{tikzpicture}[scale=1.5]
\node (B) at (-1.5,0) {$a \vee b$};
\node (E) at (-1.5,-2) {$a_*$};
\node (F) at (0.5,-1) {$b_*$};
\node (H) at (0.5,-3) {$d$};
\node (H2) at (0.5,-2) {$b'$};
\node (B1) at (-0.5,-0.5) {$a' \vee b'$};
\node (H1) at (-0.5,-2.5) {$a'$};
\node (D) at (2.5,0) {$b \vee c$};
\node (G) at (2.5,-2) {$c_*$};
\node (DG) at (2.5,-1) {$b' \vee c_*$};

\draw (B1) -- (H1);
\draw (H2) -- (DG);
\draw (B) -- (E) -- (H1) -- (H) -- (H2) -- (F) -- (B1) -- (B);
\draw (F) -- (D) -- (DG) -- (G) -- (H);
\end{tikzpicture}
\end{center}

If $(a' \vee b_*) \vee (b \vee c) < a \vee b \vee c$, then consider the following. The inequalities $a' < a_* < a \vee b$, $b_* < b \vee c < (a' \vee b_*) \vee (b \vee c)$, $a' < (a' \vee b_*) \vee (b \vee c)$ and $b_* < a \vee b$ hold. Moreover, the equalities $a_* \wedge (b \vee c) = a' \wedge b_* = d$ and $a_* \vee (b \vee c) = (a \vee b) \vee ((a' \vee b_*) \vee (b \vee c)) = a \vee b \vee c$ hold. Furthermore, $a_* \parallel b_*$ and $a \vee b \parallel b \vee c$. If $a' \geq b \vee c$, then $a_* \geq a' \geq b \vee c \geq b_*$, which is impossible. So $a' \ngeq b \vee c$. If $a' \leq b \vee c$, then $a' \leq a_* \wedge (b \vee c) = d$, which is impossible. So $a' \ngeq b \vee c$, and it follows that $a' \parallel b \vee c$. If $a_* \geq (a' \vee b_*) \vee (b \vee c)$, then $a_* \geq (a' \vee b_*) \vee (b \vee c) \geq b_*$, which is impossible. So $a_* \ngeq (a' \vee b_*) \vee (b \vee c)$. If $a_* \leq (a' \vee b_*) \vee (b \vee c)$, then $a_* \vee b_* \vee c_* \leq a_* \vee (b \vee c) \leq a_* \vee ((a' \vee b_*) \vee (b \vee c)) = (a' \vee b_*) \vee (b \vee c) < a \vee b \vee c$. But that is impossible, so $a_* \nleq (a' \vee b_*) \vee (b \vee c)$. Hence, $a_* \parallel (a' \vee b_*) \vee (b \vee c)$. It follows that $a'$, $a_*$, $a \vee b$, $b_*$, $b \vee c$, and $(a' \vee b_*) \vee (b \vee c)$ satisfy the hypothesis of Lemma \ref{L15}. This is depicted below. But then, by Lemma \ref{L15}, $L$ contains a sublattice isomorphic to $L_{15}$. That is impossible because $L_{15} \notin \mathcal{N}$.
\begin{center}
\begin{tikzpicture}[scale=1.5]
\node (B) at (-1.5,0) {$a \vee b$};
\node (E) at (-1.5,-2) {$a_*$};
\node (F) at (0.5,-1) {$b_*$};
\node (H) at (0.5,-3) {$d$};
\node (H2) at (0.5,-2) {$b'$};
\node (B1) at (-0.5,-0.5) {$a' \vee b_*$};
\node (H1) at (-0.5,-2.5) {$a'$};
\node (D) at (2.5,0) {$b \vee c$};
\node (G) at (2.5,-2) {$c_*$};
\node (DG) at (2.5,-1) {$b' \vee c_*$};
\node (T) at (0.5,1) {$a \vee b \vee c$};
\node (TT) at (1.5, 0.5) {$(a' \vee b_*) \vee (b \vee c)$};

\draw (B1) -- (TT) -- (D);
\draw (B) -- (T) -- (TT);

\draw (B1) -- (H1);
\draw (H2) -- (DG);
\draw (B) -- (E) -- (H1) -- (H) -- (H2) -- (F) -- (B1) -- (B);
\draw (F) -- (D) -- (DG) -- (G) -- (H);
\end{tikzpicture}
\end{center}

If $a' \vee b' < a' \vee b_*$ and $(a' \vee b_*) \vee (b \vee c) = a \vee b \vee c$, then we can assume without loss of generality that $a_* \wedge (a' \vee b_*) = a'$ by setting $a' = a_* \wedge (a' \vee b_*)$. This is because $a' \vee b_* < a_* \vee b_*$, implying that $a_* > a_* \wedge (a' \vee b_*) > d$, and because $a' \vee b_* \leq (a_* \wedge (a' \vee b_*)) \vee b_* \leq a' \vee b_*$, implying that $(a_* \wedge (a' \vee b_*)) \vee b_* = a' \vee b_*$. Let $b'' = (a' \vee b') \wedge b_*$. Then, as indicated in the following diagram, the sublattice of $L$ generated by 
$$ \{a_*, a', d, a' \vee b', b'', a \vee b, a' \vee b_*, b_*, a \vee b \vee c, b \vee c\} $$ 
is isomorphic to $L_{12}$. But that is impossible because $L_{12} \notin \mathcal{N}$. So, it is impossible for $a' \vee b_* < a_* \vee b_*$ and $b' \vee c_* < b_* \vee c_*$. \\

\begin{center}
\begin{tikzpicture}[scale=1.5]
\node (B) at (-1.5,0) {$a \vee b$};
\node (E) at (-1.5,-2) {$a_*$};
\node (F) at (0.5,-1) {$b_*$};
\node (H) at (0.5,-3) {$d$};
\node (H2) at (0.5,-2.2) {$b'$};
\node (B1) at (-0.5,-0.5) {$a' \vee b_*$};
\node (H1) at (-0.5,-2.5) {$a'$};
\node (D) at (2.5,0) {$b \vee c$};
\node (G) at (2.5,-2) {$c_*$};
\node (DG) at (2.5,-1.2) {$b' \vee c_*$};
\node (H3) at (0.5, -1.6) {$b''$};
\node (B3) at (-0.5,-1.1) {$a' \vee b'$};
\node (T) at (0.5,1) {$a \vee b \vee c$};

\draw (B3) -- (H3);
\draw (B1) -- (B3) -- (H1);
\draw (H2) -- (DG);
\draw (B) -- (E) -- (H1) -- (H) -- (H2) -- (H3) -- (F) -- (B1) -- (B);
\draw (F) -- (D) -- (DG) -- (G) -- (H);
\draw (B) -- (T) -- (D);
\end{tikzpicture}
\end{center}

Therefore, from the above analysis, it follows that $a$ covers $d$ in $L$, $b$ covers $d$ in $L$, or $c$ covers $d$ in $L$. \\

Secondly, assume that $Y = \{a, b, c\}$ is an antichain in $L$ for some distinct elements $a,b,c \in L$ and assume that $a \vee b = b \vee c = c \vee a = d$ for some $d \in L$. Since $L_2$ is not semidistributive, since $L_{11} \notin \mathcal{N}$, and since $L_{14} \notin \mathcal{N}$, the above arguments imply by symmetry that $a$ is covered by $d$ in $L$, $b$ is covered by $d$ in $L$, or $c$ is covered by $d$ in $L$. \\

Thirdly, assume that $Y$ is an antichain in $L$ such that for some element $d \in L$, $a \wedge b = d$ for all distinct elements $a, b \in Y$. Suppose that $|Y| \geq 4$. Then let $a_1, a_2, a_3, a_4 \in Y$ be distinct elements. Consider the elements $a_2 \vee a_3$ and $(a_1 \vee a_2 \vee a_3) \wedge (a_2 \vee a_3 \vee a_4)$. By assumption, $a_2 \parallel a_3$. Suppose that $a_1 \vee a_2 \vee a_3 \geq a_2 \vee a_3 \vee a_4$. Using the above analysis for $\{a_1, a_2, a_3\}$
, and $\{a_2, a_3, a_4\}$, we see that
$$\{d, (a_1 \vee a_2) \wedge (a_3 \vee a_1), (a_1 \vee a_2) \wedge (a_2 \vee a_3), (a_3 \vee a_1) \wedge (a_2 \vee a_3), a_1 \vee a_2, a_3 \vee a_1, a_2 \vee a_3, a_1 \vee a_2 \vee a_3 \}$$
and
$$\{d, (a_2 \vee a_3) \wedge (a_4 \vee a_2), (a_2 \vee a_3) \wedge (a_3 \vee a_4), (a_4 \vee a_2) \wedge (a_3 \vee a_4), a_2 \vee a_3, a_4 \vee a_2, a_3 \vee a_4, a_2 \vee a_3 \vee a_4 \}$$
are sublattices of $L$ that are isomorphic to $\textbf{2} \times \textbf{2} \times \textbf{2}$. Hence, as $(a_2 \vee a_4) \wedge (a_3 \wedge a_4) \geq a_4$, $(a_2 \vee a_3) \wedge a_4 = d$. By assumption, $a_1 \wedge a_4 = d$. Hence, as $L$ is semidistributive, $(a_1 \vee a_2 \vee a_3) \wedge a_4 = (a_1 \vee (a_2 \vee a_3)) \wedge a_4 = d$. But that is impossible as, by our supposition, $a_1 \vee a_2 \vee a_3 \geq a_2 \vee a_3 \vee a_4 = a_4$, implying that $(a_1 \vee a_2 \vee a_3) \wedge a_4 = a_4$. So $a_1 \vee a_2 \vee a_3 \ngeq a_2 \vee a_3 \vee a_4$. By symmetry, $a_2 \vee a_3 \vee a_4 \ngeq a_1 \vee a_2 \vee a_3$. Hence, $a_1 \vee a_2 \vee a_3 \parallel a_2 \vee a_3 \vee a_4$. Therefore, because $L$ has no doubly reducible elements, $a_2 \vee a_3 < (a_1 \vee a_2 \vee a_3) \wedge (a_2 \vee a_3 \vee a_4)$. It follows, as $a_1 \vee a_2 \vee a_3 \parallel a_2 \vee a_3 \vee a_4$, that
$$a_2 \vee a_3 < (a_1 \vee a_2 \vee a_3) \wedge (a_2 \vee a_3 \wedge a_4) < a_2 \vee a_3 \vee a_4.$$
Hence, $a_2 \vee a_3 \vee a_4$ does not cover $a_2 \vee a_3$ in $L$. The situation is depicted by the following poset, where $b_1 = a_1 \vee a_2$, $b_2 = a_1 \vee a_3$, $b_3 = a_2 \vee a_4$, $b_4 = a_3 \vee a_4$, $c_1 = a_2 \vee a_3$, $c_2 = (a_1 \vee a_2 \vee a_3) \wedge (a_2 \vee a_3 \vee a_4)$, $d_1 = a_1 \vee a_2 \vee a_3$, and $d_2 = a_2 \vee a_3 \vee a_4$.

\begin{center}
\begin{tikzpicture}[scale=1.2]
\node (d) at (0,0) {$d$};
\node (a1) at (-2,1) {$a_1$};
\node (a2) at (-1,1) {$a_2$};
\node (a3) at (1,1) {$a_3$};
\node (a4) at (2,1) {$a_4$};
\node (b1) at (-2,2) {$b_1$};
\node (b2) at (-1,2) {$b_2$};
\node (c1) at (0,2) {$c_1$};
\node (b3) at (1,2) {$b_3$};
\node (b4) at (2,2) {$b_4$};
\node (c2) at (0,2.7) {$c_2$};
\node (d1) at (-1.5,3.2) {$d_1$};
\node (d2) at (1.5,3.2) {$d_2$};

\draw (a1) -- (d) -- (a4) -- (b4) -- (d2) -- (c2) -- (d1) -- (b1) -- (a1) -- (b2) -- (a3) -- (b4);
\draw (a1) -- (b2);
\draw (b3) -- (a4);
\draw (b3) -- (a2) -- (c1) -- (a3);
\draw (c1) -- (c2);
\draw (b1) -- (a2);
\draw (d1) -- (b2);
\draw (d2) -- (b3);
\draw (a2) -- (d);
\draw (a3) -- (d);
\end{tikzpicture}
\end{center}

By symmetry, $a_2 \vee a_3 \vee a_4$ does not cover $a_3 \vee a_4$ in $L$ and $a_2 \vee a_3 \vee a_4$ does not cover $a_4 \vee a_2$ in $L$. But as
$$(a_2 \vee a_3) \vee (a_3 \vee a_4) = (a_3 \vee a_4) \vee (a_4 \vee a_2) = (a_4 \vee a_2) \vee (a_2 \vee a_3) = a_2 \vee a_3 \vee a_4,$$
this contradicts the previous analysis. Hence, $|Y| \leq 3$. \\

Lastly, assume that $Y$ is an antichain in $L$ such that for some element $d \in L$, $a \vee b = d$ for all distinct elements $a, b \in Y$. Then the above arguments imply, by symmetry, that $|Y| \leq 3$. This completes the proof.
\end{proof}

\begin{remark} Note that $\mathcal{N}$ is a semidistributive variety, so every member of $\mathcal{N}$ is semidistributive. Moreover, in the proof of Theorem \ref{cube theorem}, we only made essential use of the assumptions that $L \in \mathcal{N}$ and that $L$ satisfies Whitman's condition. Hence, Theorem \ref{cube theorem} describes a property of lattices $L \in \mathcal{N}$ that satisfy Whitman's condition.
\end{remark}

\section{Perspective properties}\label{sec:perpective properties}

A well-known characterization of modular lattices \cite{ILO} states that a lattice $M$ is modular if and only if for all $a, b \in M$ satisfying $a \parallel b$, the maps $j_b : x \mapsto x \vee b$ and $m_a : y \mapsto y \wedge a$ are mutually inverse lattice isomorphisms between the following invervals in $M$, $[a \wedge b, a]$ and $[b, a \vee b]$. Moreover, as noted in Section \ref{sec:background}, a sublattice of a free lattice is modular if and only if it is distributive. \\

Motivated by these considerations, we prove the second and third main results of this paper by proving two structural properties of sublattices of a free lattice that are in the variety generated by $N_5$. They illustrate that if these sublattices of free lattice are structurally complex, then they satisfy properties that are similar to the above characterization of modular lattices and to the modular law in Definition \ref{modular} \\

A \emph{convex distributive sublattice} of a lattice $L$ is a lattice $K$ such that $K$ is a distributive sublattice of $L$ and $K$ is a convex sublattice of $L$. With this, we introduce the following notions. \\

\begin{definition} Let $K$ be a lattice. Then call a set partition $\mathcal{F}$ of the set of elements of $K$ \emph{distributive} if every element $F \in \mathcal{F}$ is a convex distributive sublattice of $K$ and, for all distinct elements $F_1, F_2 \in \mathcal{F}$, $F_1 \cup F_2$ is not a sublattice of $K$, $F_1 \cup F_2$ is not a convex subset of $K$, or $F_2 \cup F_2$ is a convex distributive sublattice of $K$.
\end{definition}

Every lattice $L$ has the distributive partition $\mathcal{F} = \{ \{a\} : a \in L \}$, as for all distinct elements $a, b \in L$, $\{a, b\}$ is an antichain or $\{a, b\}$ is the two element chain. However, in general, $\mathcal{F} = \{ \{a\} : a \in L \}$ is not the only distributive partition of a lattice $L$.

\begin{definition} Let $K$ be a lattice. Then define $\Dec(K)$ to be the minimum possible cardinality of a distributive partition of $K$.
\end{definition}

As every lattice $L$ has the distributive partition $\mathcal{F} = \{ \{a\} : a \in L \}$, $\Dec(L) \leq |L|$. A lattice $L$ satisfies $\Dec(L) = 1$ if and only if $L$ is distributive. This is because $L$ is distributive if and only if $\{L\}$ is distributive partition of $L$. Moreover, it is impossible for $\Dec(L) = 2$ because $\Dec(L) = 2$ implies that there is a distributive partition $\{L_1, L_2\}$ of $L$ such that $L_1$ is distributive, $L_2$ is distributive, but $L_1 \cup L_2 = L$ is a non-distributive sublattice of $L$, which is impossible. We give two examples of lattices $L$ that satisfy $\Dec(L) \geq 3$.

\begin{example} Consider the lattice $N_5$ with elements labelled below.
\begin{center}
\begin{tikzpicture}[scale=0.5]
\node (x1) at (0,0) {$x_1$};
\node (x2) at (-3,-3) {$x_2$};
\node (x3) at (2,-2) {$x_3$};
\node (x4) at (2,-4) {$x_4$};
\node (x5) at (0,-6) {$x_5$};

\draw (x1) -- (x2) -- (x5) -- (x4) -- (x3) -- (x1);
\end{tikzpicture}
\end{center}
This lattice satisfies $\Dec(N_5) = 3$. In fact, there are six distributive partitions $\mathcal{F}$ of $N_5$ that satisfy $|\mathcal{F}| = 3$. They are $\{\{x_1,x_2\}$, $\{x_3\}$, $\{x_4,x_5\}\}$, $\{\{x_1,x_3\}$, $\{x_4\},$ $\{x_2,x_5\}\}$, $\{\{x_1,x_2\},$ $\{x_3,x_4\},$ $\{x_5\}\}$, $\{\{x_1\}$, $\{x_2, x_5\}$, $\{x_3,x_4\} \}$, $\{\{x_1,x_3,x_4\}$, $\{x_2\}$, $\{x_5\}\}$, and $\{\{x_1\}$, $\{x_2\}$, $\{x_3,x_4,x_5\} \}$. To see that $\{\{x_1,x_2\}$, $\{x_3\}$, $\{x_4,x_5\}\}$ is a distributive partition of $\mathcal{F}$, note that $\{x_1, x_2\}$, $\{x_3\}$, and $\{x_4, x_5\}$ are all convex distributive sublattices of $N_5$. Moreover, note that $\{x_1, x_2\} \cup \{x_3\}$ is not a sublattice of $N_5$, $\{x_1, x_2\} \cup \{x_4, x_5\}$ is not a convex subset of $N_5$, and $\{x_3\} \cup \{x_4, x_5\}$ is a convex distributive sublattice of $L$. \\

Next, consider the lattice $L$ depicted below.
\begin{center}
\begin{tikzpicture}[scale=0.4]
\node (x1) at (0,0) {$\circ$};
\node (x2) at (-3,-3) {$\circ$};
\node (x3) at (2,-2) {$\circ$};
\node (x4) at (2,-4) {$\circ$};
\node (x5) at (0,-6) {$\circ$};

\node (y1) at (0,-8) {$\circ$};
\node (y2) at (-3,-11) {$\circ$};
\node (y3) at (2,-10) {$\circ$};
\node (y4) at (2,-12) {$\circ$};
\node (y5) at (0,-14) {$\circ$};

\draw (x1) -- (x2) -- (x5) -- (x4) -- (x3) -- (x1);

\draw (y1) -- (y2) -- (y5) -- (y4) -- (y3) -- (y1);

\draw (x5) -- (y1);
\end{tikzpicture}
\end{center}
The lattice $L$ is isomorphic to a sublattice of a free lattice, and $L \in \mathcal{N}$. Moreover, it can be checked that $\Dec(L) = 5$.
\end{example}

\begin{remark} For sublattices $L$ of free lattices, the quantity $\Dec(L)$ gives one measure of the structural complexity of $L$ for the following reason. Because a distributive partition of a lattice $L$ is a decomposition of $L$ into \emph{convex} distributive sublattices, and because, by Theorem \ref{Galvin and Jonsson}, all distributive sublattices of free lattices have a very simple structure, it follows that if $L$ is a sublattice of a free lattice and if $\Dec(L)$ is small, then the structure of $L$ is relatively simple.
\end{remark}

We now state, and prove, the second structural property of this paper. 

\begin{theorem}\label{partial modularity 1} Let $L \in \mathcal{N}$, and assume that $L$ is isomorphic to a sublattice of a free lattice. Moreover, assume that there exists a sublattice $K$ of $L$ and an element $a \in L$ such that for all $b \in K$, $a \parallel b$. Then
$$\Dec(K) \leq |\{ a \vee b : b \in K \} \times \{ a \wedge b : b \in K \}|.$$
\end{theorem}

In order to prove Theorem \ref{partial modularity 1}, we prove Lemma \ref{degeneracy lemma}. Recall that a \emph{convex sublattice} of $L$ is a sublattice of $L$ that is also a convex subset of $L$.

\begin{lemma}\label{degeneracy lemma} Let $L \in \mathcal{N}$, and assume that $L$ is a sublattice of a free lattice. Then the following property holds. For any convex sublattice $K$ of $L$, if $a \in L$ satisfies $a \parallel b$ for all $b \in K$, then $|\{a \vee b : b \in K \}| \geq 3$, $|\{a \wedge b : b \in K \}| \geq 3$, or $K$ is a distributive lattice.
\end{lemma}

\begin{proof} As $a \parallel b$ for all $b \in K$, it follows that for all $c \in \{a \vee b : b \in K\}$, $c \notin K$, for all $d \in \{a \wedge b : b \in K\}$, $d \notin K$, and $\{a \wedge b : b \in K \} \cap \{a \vee b : b \in K \} = \emptyset$. It is enough to assume without loss of generality that $K$ is not distributive. By the $M_3-N_5$ Theorem, $K$ has a sublattice $N$ that is isomorphic to $N_5$. Label the elements of $N$ as follows.

\begin{center}
\begin{tikzpicture}[scale=0.5]
\node (x1) at (0,0) {$x_1$};
\node (x2) at (-3,-3) {$x_2$};
\node (x3) at (2,-2) {$x_3$};
\node (x4) at (2,-4) {$x_4$};
\node (x5) at (0,-6) {$x_5$};

\draw (x1) -- (x2) -- (x5) -- (x4) -- (x3) -- (x1);
\end{tikzpicture}
\end{center}

Suppose that $|\{a \vee b : b \in K \}| \leq 2$ and that $|\{a \wedge b : b \in K \}| \leq 2$. Then $|\{a \vee b : b \in N \}| \leq 2$ and $|\{a \wedge b : b \in N \}| \leq 2$. Since $a \parallel b$ for all $b \in K$, $c \notin K$ for all $c \in \{a \vee b : b \in K \}$ and $d \notin K$ for all $d \in \{a \wedge b : b \in K \}$. There are four main cases to consider. \\

For the first main case, suppose that $|\{a \vee b : b \in N\}| = 1$ and that $|\{a \wedge b : b \in N\}| = 1$. Let $a_1$ be the element of $\{a \vee b : b \in N\}$ and let $a_0$ be the element of $\{a \wedge b : b \in N\}$. Then for all $b \in N$, $b \vee a = a_1$ and $b \wedge a = a_0$. Hence, the sublattice of $L$ generated by $N \cup \{a\}$, depicted below, is isomorphic to $L_6$. But that is impossible because $L_6 \notin \mathcal{N}$.

\begin{center}
\begin{tikzpicture}
\node (x1) at (0,0) {$x_1$};
\node (x2) at (-1,-1) {$x_2$};
\node (x3) at (1,-0.6) {$x_3$};
\node (x4) at (1,-1.4) {$x_4$};
\node (x5) at (0,-2) {$x_5$};

\node (x1a) at (1.2,1.1) {$x_5 \vee a$};
\node (a) at (3.2,-1) {$a$};
\node (bot) at (1.2,-3.1) {$x_1 \wedge a$};

\draw (x1) -- (x2) -- (x5) -- (x4) -- (x3) -- (x1);
\draw (x1) -- (x1a) -- (a) -- (bot) -- (x5);
\end{tikzpicture}
\end{center}

For the second main case, suppose that $|\{a \vee b : b \in N\}| = 2$ and $|\{a \wedge b : b \in N\}| = 1$. If $(x_5 \vee a) \wedge x_1 \geq x_2$ and $(x_5 \vee a) \wedge x_1 \geq x_4$, then, as $x_2 \vee x_4 = x_1$, $(x_5 \vee a) \wedge x_1 \geq x_1$, implying that $x_5 \vee a \geq x_1$. But then, $|\{a \vee b : b \in N\}| = 1$, contrary to the supposition that $|\{a \vee b : b \in N \}| = 2$. Similarly, $(x_5 \vee a) \wedge x_1 \ngeq x_1$. So it is enough to consider the following. \\

If $x_4 < (x_5 \vee a) \wedge x_1 \leq x_3$, then $x_5 \vee a < x_1 \vee a$. Let $x_3' = (x_5 \vee a) \wedge x_1$. As $x_4 < x_3' \leq x_3$, the sublattice of $L$ generated by $\{x_1, x_2, x_3', x_4, x_5\}$ is isomorphic to $N_5$. So as $x_2 \vee x_4 = x_1$, $x_2 \nleq x_5 \vee a$. Hence, because $|\{a \vee b : b \in N\}| = 2$, it follows that $x_2 \vee a = x_1 \vee a$. Moreover, as $x_2 \leq x_1$, $x_2 \wedge (x_5 \vee a) \leq x_2 \wedge x_1 \wedge (x_5 \vee a) = x_2 \wedge x_3' = x_5$. Furthermore, as $|\{a \wedge b : b \in N \}| = 1$, $a \wedge x_3 = a \wedge x_5$. Hence, as $x_3 \geq x_3' \geq x_5$, $a \wedge x_3' = a \wedge x_5$. Lastly, as $x_5 < x_3' < a \vee x_5$, $a \vee x_3' = a \vee x_5$. It follows that the sublattice of $L$ generated by $\{x_1, x_2, x_3', x_4, x_5, x_1 \vee a, x_5 \vee a, x_1 \wedge a \}$, depicted below, is isomorphic to $L_7$. But that is impossible because $L_7 \notin \mathcal{N}$.

\begin{center}
\begin{tikzpicture}
\node (x1) at (0,0) {$x_1$};
\node (x2) at (-1,-1) {$x_2$};
\node (x3) at (1,-0.6) {$x_3'$};
\node (x4) at (1,-1.4) {$x_4$};
\node (x5) at (0,-2) {$x_5$};

\node (x1a) at (1.2,1.1) {$x_2 \vee a$};
\node (x3a) at (2.2,0) {$x_5 \vee a$};
\node (a) at (3.2,-1.2) {$a$};
\node (bot) at (1.2,-3.1) {$x_1 \wedge a$};

\draw (x1) -- (x2) -- (x5) -- (x4) -- (x3) -- (x1);
\draw (x1a) -- (x3a) -- (x3) -- (x3a) -- (a) -- (bot) -- (x5);
\draw (x1) -- (x1a);
\end{tikzpicture}
\end{center}

If $x_2 < (x_5 \vee a) \wedge x_1 < x_1$ or $x_3 < (x_5 \vee a) \wedge x_1 < x_1$, then assume without loss of generality that $x_3 < (x_5 \vee a) \wedge x_1 < x_1$. Define $x_3' = (x_5 \vee a) \wedge x_1$. Note that $x_3' > x_3$. If $x_2 \wedge x_3' = x_5$, then the sublattice of $L$ generated by $\{x_1, x_2, x_3', x_3, x_5\}$ is isomorphic to $N_5$. Moreover, as $|\{a \wedge b : b \in N\}| = 1$, $x_1 \wedge a = x_4 \wedge a$. So as $x_1 < x_3' < x_4$, it follows that $x_3' \wedge a = x_1 \wedge a$. Furthermore, as $x_5 < x_3' < x_5 \vee a$, $a \vee x_3' = x_5 \vee a$. From this, we see that if $x_2 \wedge x_3' = x_5$, then, because $x_3 < (x_5 \vee a) \wedge x_1 \leq x_3'$, we can proceed as before to derive a contradiction. So assume that $x_2 \wedge x_3' > x_5$, and set $x_3'' = x_2 \wedge x_3'$. Define $x_3'''  = x_3'' \vee x_3$. If $x_3'' \geq x_3$, then $x_2 \geq x_3'' \geq x_3$, which is impossible. Moreover, if $x_3'' \leq x_3$, then $x_3'' \leq x_2 \wedge x_3 = x_5$, which is impossible. So $x_3'' \parallel x_3$. Since $|\{a \vee b : b \in N\}| > 1$, $x_1 \nleq x_5 \vee a$ and, as $a \nleq x_1$, it is clear that $x_1 \ngeq x_5 \vee a$. Hence, $x_1 \parallel x_5 \vee a$. So, since $L$ has no doubly reducible elements, it follows that $x_3''' < x_3'$. By definition, $x_3'' = x_2 \wedge x_3'$. Moreover, $x_1 \geq x_2 \vee x_3''' \geq x_2 \vee x_3 = x_1$, implying that $x_2 \vee x_3''' = x_1$. Hence, the sublattice of $L$ generated by $\{x_1, x_2, x_3', x_3''', x_3''\}$ is isomorphic to $N_5$. Since $|\{a \wedge b : b \in N\}| = 1$, $x_1 \wedge a = x_5 \wedge a$. So as $x_1 > x_3' > x_3''' > x_3'' > x_5$, $x_3' \wedge a = x_3'' \wedge a = x_3''' \wedge a = x_1 \wedge a$. Moreover, $x_5 < x_3'' < x_3''' < x_3' < x_5 \vee a$, so $x_3'' \vee a = x_3''' \vee a = x_3' \vee a = x_5 \vee a$. The resulting construction is depicted below. Because $x_3''' < (x_5 \vee a) \wedge x_1 \leq x_3'$, it follows that we can proceed as before to derive a contradiction.

\begin{center}
\begin{tikzpicture}[scale=1.2]
\node (x1) at (0,0) {$x_1$};
\node (x2) at (-1,-1) {$x_2$};
\node (x3) at (1,-0.6) {$x_3'$};
\node (x333) at (1,-1.4) {$x_3'''$};
\node (x33) at (0,-2) {$x_3''$};

\node (x1a) at (1.2,1.1) {$x_1 \vee a$};
\node (x3a) at (2.2,0) {$x_5 \vee a$};
\node (a) at (4,-2) {$a$};
\node (bot) at (2,-4) {$x_1 \wedge a$};

\node (x4) at (2,-2.2) {$x_3$};
\node (x5) at (1,-3) {$x_5$};

\draw (x1) -- (x2) -- (x33) -- (x333) -- (x3) -- (x1);
\draw (x1a) -- (x3a) -- (x3) -- (x3a) -- (a) -- (bot) -- (x5) -- (x33);
\draw (x1) -- (x1a);
\draw (x333) -- (x4) -- (x5);
\end{tikzpicture}
\end{center}

If $(x_5 \vee a) \wedge x_1 = x_2$, then, as $x_4 \leq x_1$, $x_4 \wedge (x_5 \vee a) = x_4 \wedge x_1 \wedge (x_5 \vee a) = x_4 \wedge x_2 = x_5$ and, as $x_3 \leq x_1$, $x_3 \wedge (x_5 \vee a) = x_3 \wedge x_1 \wedge (x_5 \vee a) = x_3 \wedge x_2 = x_5$. Lastly, since $|\{a \vee b : b \in N\}| > 1$ and $x_4 \vee x_2 = x_1$, $x_4 \nleq x_5 \vee a$. Hence, because $|\{a \vee b : b \in N \}| = 2$, it follows that $x_4 \vee a = x_1 \vee a$. From this, it can be seen that the sublattice of $L$ generated by $\{a, x_1, x_2, x_3, x_4, x_5 \}$ is isomorphic to $L_9$, this is depicted below. But that is impossible because $L_9 \notin \mathcal{N}$.

\begin{center}
\begin{tikzpicture}[scale=2.2]
\node (x1) at (0,0.2) {$x_1$};
\node (x2) at (0.55,-0.4) {$x_2$};
\node (x3) at (-0.55,-0.2) {$x_3$};
\node (x4) at (-1.1,-0.6) {$x_4$};
\node (x5) at (-0.55,-1.2) {$x_5$};

\node (x1a) at (0.6,0.65) {$x_4 \vee a$};
\node (x3a) at (1.15,0.05) {$x_5 \vee a$};
\node (a) at (1.7,-0.6) {$a$};
\node (bot) at (0, -1.8) {$x_1 \wedge a$};

\draw (x1) -- (x2) -- (x5) -- (x4) -- (x3) -- (x1);
\draw (x1a) -- (x3a) -- (a) -- (bot) -- (x5);
\draw (x1) -- (x1a);
\draw (x2) -- (x3a);
\end{tikzpicture}
\end{center}

If $x_2 \nleq (x_5 \vee a) \wedge x_1$ and $x_3 \nleq (x_5 \vee a) \wedge x_1$, then, as $x_2 \leq x_1$ and $x_3 \leq x_1$, $x_2 \nleq x_5 \vee a$ and $x_3 \nleq x_5 \vee a$. So, because $|\{a \vee b : b \in N\}| = 2$, it follows that $x_2 \vee (x_5 \vee a) = x_3 \vee (x_5 \vee a) = x_1 \vee a$. Hence, as $L$ is semidistributive, $(x_5 \vee a) \vee (x_2 \wedge x_3) = x_1 \vee a$. But as $x_2 \wedge x_3 = x_5$ and $x_5 < x_5 \vee a$, that is impossible. \\

For the third main case, suppose that $|\{a \vee b : b \in N\}| = 1$ and $|\{a \wedge b : b \in N\}| = 2$. Because $L_8 \notin \mathcal{N}$ and because $L_{10} \notin \mathcal{N}$, it follows that, we can argue similarly to before to derive a contradiction. \\

For the fourth main case, suppose that $|\{a \vee b : b \in N\}| = 2$ and $|\{a \wedge b : b \in N\}| = 2$. Let $c$ be the greatest element of $\{a \vee b : b \in N\}$, let $c'$ be the least element of $\{a \vee b : b \in N\}$, let $d$ be the least element of $\{a \wedge b : b \in N\}$, and let $d'$ be the greatest element of $\{a \wedge b : b \in N\}$. Since $a \parallel b$ for all $b \in N$, $b \nleq d'$ for all $b \in N$ and $b \ngeq c'$ for all $b \in N$. First assume that there exists an element $b \in N$ such that $b \parallel c'$ and $b \parallel d'$. Consider the elements $x_5$, $b$, $x_1$, $d'$, $a$, and $c'$. By definition, $x_5 < b < x_1$ and $d' < a < c'$. Because $a \parallel e$ for all $e \in N$, $a \parallel b$, $a \parallel x_1$, and $a \parallel x_5$. Because $\{a \vee e : e \in N \} = \{c, c'\}$ and $x_5$ is the minimum element of $N$, $a \vee x_5 = c'$. Moreover, because $\{a \wedge e : e \in N \} = \{d, d'\}$ and $x_1$ is the maximum element of $N$, $a \wedge x_1 = d'$. As $b \parallel c'$ and $\{a \vee e : e \in N \} = \{c, c'\}$, $b \vee a = x_1 \vee c' = c$. Lastly, as $b \parallel d'$ and $\{a \wedge e : e \in N \} = \{d, d'\}$, $b \wedge a = x_5 \wedge d' = d$. This is depicted below.

\begin{center}
\begin{tikzpicture}[scale=1.2]
\node (b) at (0,0) {$b$};
\node (x1) at (1,1) {$x_1$};
\node (c) at (2,2) {$c$};
\node (cc) at (3,1) {$c'$};
\node (a) at (4,0) {$a$};
\node (dd) at (3,-1) {$d'$};
\node (d) at (2,-2) {$d$};
\node (x5) at (1,-1) {$x_5$};

\draw (b) -- (x1) -- (c) -- (cc) -- (a) -- (dd) -- (d) -- (x5) -- (b);
\draw (x1) -- (dd);
\draw (x5) -- (cc);
\end{tikzpicture}
\end{center}

So, as $L$ has no doubly reducible elements, $x_5$, $b$, $x_1$, $d'$, $a$, and $c'$ satisfy the conditions of Lemma \ref{L15}. But then, $L$ contains a sublattice that is isomorphic to $L_{15}$, which is impossible because $L_{15} \notin \mathcal{N}$. \\

Lastly, assume that for all $b \in N$, $b \leq c'$ or $b \geq d'$. If $d' \leq x_3$ and $d' \leq x_2$, then, as $x_5 = x_3 \wedge x_2$, $d' \leq x_5$. Similarly, if $c' \geq x_4$ and $c' \geq x_2$, then, as $x_1 = x_4 \vee x_2$, $c' \geq x_1$. However, $d' \leq x_5$ is impossible as $d \in \{a \wedge b : b \in N\}$, and $c' \geq x_1$ is impossible as $c \in \{a \vee b : b \in N \}$. So, assume without loss of generality that $d' \leq x_4$, $d' \parallel x_2$, $c' \geq x_2$, $c' \parallel x_3$, and $c' \parallel x_4$. This is depicted by the left-most diagram shown below.

\begin{center}
\begin{tikzpicture}[scale=2.2]
\node (x1) at (0,0.2) {$x_1$};
\node (x2) at (0.55,-0.4) {$x_2$};
\node (x3) at (-0.55,-0.2) {$x_3$};
\node (x4) at (-1.1,-0.6) {$x_4$};
\node (x5) at (-0.55,-1.2) {$x_5$};

\node (x1a) at (0.6,0.65) {$c$};
\node (x3a) at (1.15,0.05) {$c'$};
\node (a) at (1.7,-0.6) {$a$};
\node (bot) at (0, -1.8) {$d$};
\node (f) at (0.85,-1.2) {$d'$};

\draw (x1) -- (x2) -- (x5) -- (x4) -- (x3) -- (x1);
\draw (x1a) -- (x3a) -- (a) -- (f) -- (bot) -- (x5);
\draw (x1) -- (x1a);
\draw (x2) -- (x3a);
\draw (x4) -- (f);
\end{tikzpicture} \;\;\;\;
\begin{tikzpicture}[scale=2.2]
\node (x1) at (0,0.2) {$x_1$};
\node (x2) at (0.55,-0.4) {$x_2'$};
\node (x3) at (-0.55,-0.2) {$x_3$};
\node (x4) at (-1.1,-0.6) {$x_4$};
\node (x5) at (-0.55,-1.2) {$x_5'$};

\node (x1a) at (0.6,0.65) {$c$};
\node (x3a) at (1.15,0.05) {$c'$};
\node (a) at (1.7,-0.6) {$a$};
\node (bot) at (0, -1.8) {$d$};
\node (f) at (0.85,-1.2) {$d'$};

\draw (x1) -- (x2) -- (x5) -- (x4) -- (x3) -- (x1);
\draw (x1a) -- (x3a) -- (a) -- (f) -- (bot) -- (x5);
\draw (x1) -- (x1a);
\draw (x2) -- (x3a);
\draw (x4) -- (f);
\end{tikzpicture}
\end{center}
We can also assume without loss of generality that $x_2 = x_1 \wedge c'$ and $x_5 = x_4 \wedge c'$ for the following reasons. Let $x_2' = x_1 \wedge c'$ and let $x_5' = x_4 \wedge c'$. Note that $x_4 \wedge x_2' = x_4 \wedge x_1 \wedge c' = x_4 \wedge c' = x_5'$. As $\{a \vee b : b \in K\} = \{a \vee b : b \in N\} = \{c, c'\}$, as $\{a \wedge b : b \in K\} = \{a \wedge b : b \in N\} = \{d, d'\}$, and as $K$ is a convex sublattice of $L$, it follows that $\{a \vee b : b \in N'\} = \{c, c'\}$, and $\{a \vee b : b \in N'\} \subseteq \{d, d'\}$. This is depicted by the right-most diagram shown above. \\

If $x_3 \wedge x_2' = x_5'$, then $\{x_1, x_2', x_3, x_4, x_5' \}$ generates a sublattice $N'$ of $L$ that is isomorphic to $N_5$. Hence, as $|\{a \vee b : b \in N'\}| = 2$ and $|\{a \wedge b : b \in N' \}| \leq 2$, we have reduced this subcase to a case that we have already proven to be impossible, or to the above case in which $x_2 = x_1 \wedge c'$ and $x_5 = x_4 \wedge c'$. \\

If $x_3 \wedge x_2' > x_5'$, then, as $x_2 \leq x_2'$ and $x_5 \leq x_5' \leq x_4$,
$$(x_3 \wedge x_2') \wedge x_2 = x_3 \wedge x_2 = x_5 = x_4 \wedge x_2 = x_5' \wedge x_2. $$
Moreover, as $x_3 \wedge x_2' > x_5'$, $x_5' > x_5$, implying that $x_5' \parallel x_2$. Consider the sublattice $L'$ of $L$ generated by $\{x_5', x_3 \wedge x_2', (x_3 \wedge x_2') \vee x_2, x_2, x_5\}$. This is partially depicted by the below diagram. \\

\begin{center}
\begin{tikzpicture}[scale=2.2]
\node (x1) at (0,0.2) {$(x_3 \wedge x_2') \vee x_2$};
\node (x2) at (0.55,-0.4) {$x_2$};
\node (x3) at (-0.55,-0.2) {$x_3 \wedge x_2'$};
\node (x4) at (-1.1,-0.6) {$x_5'$};
\node (x5) at (-0.55,-1.2) {$x_5$};

\node (x1a) at (0.6,0.65) {$c'$};
\node (a) at (1.7,-0.6) {$a$};
\node (bot) at (0, -1.8) {$d$};

\draw (x1) -- (x2) -- (x5) -- (x4) -- (x3) -- (x1);
\draw (x1a) -- (a) -- (bot) -- (x5);
\draw (x1) -- (x1a);
\end{tikzpicture}
\end{center}

Suppose that $L'$ does not contain a sublattice that is isomorphic to $N_5$. As $(x_3 \wedge x_2') \wedge x_2 = x_5$, $x_5' \vee x_2 < (x_3 \wedge x_2') \vee x_2$ for otherwise, $\{x_5', x_3 \wedge x_2', (x_3 \wedge x_2') \vee x_2, x_5, x_2\}$ generates a sublattice of $L'$ that is isomorphic to $N_5$. Moreover, $x_5' \vee x_2 \leq x_3 \wedge x_2'$ implies that $x_2 \leq x_3$, and $x_5' \vee x_2 > x_3 \wedge x_2'$ implies that $x_5' \vee x_2 = (x_3 \wedge x_2') \vee x_2$, so $x_5' \vee x_2 \parallel x_3 \wedge x_2$. Lastly, as $x_5' \parallel x_2$ and $((x_5' \vee x_2) \wedge (x_3 \wedge x_2')) \wedge x_2 = x_3 \wedge x_2 = x_5$, $(x_3 \wedge x_2') \wedge (x_5' \vee x_2) = x_5'$ for otherwise, $\{x_5', (x_5' \vee x_2) \wedge (x_3 \wedge x_2'), x_5' \vee x_2, x_5, x_2 \}$ generates a sublattice of $L'$ that is isomorphic to $N_5$. Hence, the following holds. Because $x_1 \geq x_4 \vee (x_5' \vee x_2) \geq x_4 \vee x_2 = x_1$ implies that $x_4 \vee (x_5' \vee x_2) = x_1$, $(x_3 \wedge x_2') \wedge (x_4 \vee (x_5' \vee x_2)) = (x_3 \wedge x_2') \wedge x_1 = x_3 \wedge x_2'$. But $L$ is semidistributive, $(x_3 \wedge x_2') \wedge x_4 = x_4 \wedge x_2' = x_5'$, and $(x_3 \wedge x_2') \wedge (x_5' \vee x_2) = x_5' < x_3 \wedge x_2'$, implying that $(x_3 \wedge x_2') \wedge (x_4 \vee (x_5' \vee x_2)) = x_5' < x_3 \wedge x_2'$. So we have reached a contradiction. \\

Hence, $L'$ contains a sublattice $N''$ that is isomorphic to $N_5$. Since $|\{a \vee b : b \in K\}| \leq 2$, $\{a \vee b : b \in K\} = \{c, c'\}$. Moreover, because $K$ is a convex sublattice of $L$, because $x_5 \leq b \leq c'$ for all $b \in L'$, and because $x_5 \vee a = c'$, 
$$\{a \vee b : b \in N'' \} = \{a \vee b : b \in L'\} = \{c'\}. $$
Since $|\{a \wedge b : b \in K\}| \leq 2$, $\{a \wedge b : b \in K\} = \{d, d'\}$. Moreover, because $K$ is a convex sublattice of $L$, it follows that
$$\{a \wedge b : b \in N'' \} \subseteq \{a \wedge b : b \in L'\} \subseteq \{d, d'\}. $$
Hence, we have reduced this subcase to a case that we have already proven to be impossible, or to the above case in which $x_2 = x_1 \wedge c'$ and $x_5 = x_4 \wedge c'$. \\

Therefore, assume without loss of generality that $x_2 = x_1 \wedge c'$ and $x_5 = x_4 \wedge c'$. We reach a contradiction as follows. Since $x_5 = x_4 \wedge c'$ and $d' \leq x_4$, the fact that $d' \leq a \leq c'$ implies that $d' \leq x_5$. But that is impossible because $\{a \wedge b : b \in N\} = \{d, d'\}$.

\end{proof}

Now, we prove Theorem \ref{partial modularity 1}.

\begin{proof} Let $I_1 = \{a \vee b : b \in K\}$, and let $\mathcal{F}_1 = \{Y_i : i \in I_1 \}$ be the set partition of the set of elements of $K$ where for all $i \in I_1$, $Y_i = \{b \in K : a \vee b = i \}$. Similarly, define $I_0 = \{a \wedge b : b \in K\}$, and let $\mathcal{F}_0 = \{X_i : i \in I_0 \}$ be the partition of $K$ where for all $i \in I_0$, $F_i = \{b \in K : a \wedge b = i \}$. For all $i \in I_1$, $Y_i$ is a convex subset of $L$, and for all $i \in I_0$, $X_i$ is a convex subset of $L$. Moreover, since $L$ is semidistributive, $Y_i$ is a sublattice of $K$ for all $i \in I_1$ and $X_i$ is a sublattice of $K$ for all $i \in I_0$. Now, consider the following set partition of the set of elements of $K$ 
$$\mathcal{F} = \{X_i \cap Y_j : X_i \cap Y_j \neq \emptyset, X_i \in \mathcal{F}_0, \text{ and } Y_i \in \mathcal{F}_1 \}.$$
The intersection of two convex subsets of $L$ is a convex subset of $L$, and the intersection of two sublattices of $L$ is a sublattice of $L$. Hence, $F$ is a convex sublattice of $L$ for all $F \in \mathcal{F}$. Moreover, for all $F \in \mathcal{F}$
$$|\{ a \vee b : b \in F \}| = |\{a \wedge b : b \in F \}| = 1.$$
So by Lemma \ref{degeneracy lemma}, $F$ is distributive.  \\

Let $F_1, F_2 \in \mathcal{F}$ be distinct. If $F_1 \cup F_2$ is not a sublattice of $K$, or if $F_1 \cup F_2$ is not a convex subset of $L$, then we are done. So assume without loss of generality that $F_1 \cup F_2$ is a convex sublattice of $K$. Since
$$|\{a \vee b : b \in F_1 \cup F_2 \}| \leq 2 $$
and
$$|\{a \wedge b : b \in F_1 \cup F_2 \}| \leq 2, $$
Lemma \ref{degeneracy lemma} implies that $F_1 \cup F_2$ is distributive. Hence, $\mathcal{F}$ is a distributive partition of $K$. The theorem now follows as
$$\Dec(K) \leq |\mathcal{F}| \leq |I_1 \times I_0| \leq |\{ a \vee b : b \in K \} \times \{ a \wedge b : b \in K \}|.$$
\end{proof}

\begin{remark} In the proof of Theorem \ref{partial modularity 1}, we only made essential use of the assumptions that $L \in \mathcal{N}$ and that $L$ has no doubly reducible elements. So Theorem \ref{partial modularity 1} describes a property of lattices $L \in \mathcal{N}$ that have no doubly reducible elements.
\end{remark}

Given Theorem \ref{partial modularity 1}, it is natural to consider the case when $|K|$ and $\Dec(K)$ are large. We will state and prove a third structural property that applies to such a case and that resembles Definition \ref{modular}. In order to provide motivation for this, we first recall the following property of semidistributive lattices. \\

Recall that the \emph{length} of a poset $P$ is the supremum of the cardinalities of all chains in $P$.

\begin{theorem} (Dilworth, J\'onsson, Kiefer, (\cite{FL}, Theorem 5.59, p. 169), \cite{Semidistributive origins 1, Semidistributive origins 2})\label{semidistributive length} A semidistributive lattice without infinite chains is finite; if it is of length $n + 1$ for some $n \in \mathbb{N}$, then it has at most $2^n$ elements.
\end{theorem}

\begin{remark} In fact, J\'onsson and Kiefer, receiving assistance from Dilworth, proved Theorem \ref{semidistributive length} for \emph{meet semidistributive} lattices and \emph{join semidistributive lattices} \cite{Semidistributive origins 1, Semidistributive origins 2, FL}.
\end{remark}

So as all sublattices of free lattices are semidistributive, it follows that if $|K|$ is large, then $K$ will have chains with a large number of elements. We now state and prove the third structural property of this paper.

\begin{theorem}\label{partial modularity 2} Let $L$ be isomorphic to a sublattice of a free lattice, and assume that $L \in \mathcal{N}$. Moreover, let $a, b_1, b_2, b_3, b_4, b_5 \in L$ be such that $a \parallel b_i$ for all $1 \leq i \leq 5$, $b_i < b_{i+1}$ for all $1 \leq i \leq 4$, and $a \vee b_i < a \vee b_{i+1}$ for all $1 \leq i \leq 4$. If $(a \vee b_4) \wedge b_5 \neq b_4$, then $(a \vee b_3) \wedge b_5$ is covered by $a \vee b_3$ in $L$.
\end{theorem}

\begin{remark} The dual of any lattice in $\mathcal{N}$ is also in $\mathcal{N}$ and the dual of a sublattice of a free lattice is a sublattice of a free lattice. Hence, the dual of Theorem \ref{partial modularity 2} also holds.
\end{remark}

Before proving Theorem \ref{partial modularity 2}, we prove a lemma and define certain ordered pairs of sequences.

\begin{lemma}\label{12 element subposet} Let $L \in \mathcal{N}$, and assume that $L$ has no doubly reducible elements. Then, it is impossible for there to be a twelve-element subposet of $L$ as depicted below where the sublattice of $L$ generated by $\{x', x, b, z, z', w', w, a, y, y'\}$ is isomorphic to $\textbf{2} \times \textbf{5}$.
\begin{center}
\begin{tikzpicture}[scale=0.8]
\node (c) at (1,1) {$c$};
\node (x') at (-4,0) {$x'$};
\node (x) at (-2.5,1.5) {$x$};
\node (b) at (-1,3) {$b$};
\node (z) at (1,5) {$z$};
\node (z') at (2.5,6.5) {$z'$};
\node (w') at (0,-4) {$w'$};
\node (w) at (1.5,-2.5) {$w$};
\node (a) at (3,-1) {$a$};
\node (s) at (4,0) {$s$};
\node (y) at (5,1) {$y$};
\node (y') at (6.5,2.5) {$y'$};

\draw (x') -- (x) -- (b) -- (z) -- (z') -- (y') -- (y) -- (s) -- (a) -- (w) -- (w') -- (x');
\draw (x) -- (w);
\draw (b) -- (c) -- (a);
\draw (z) -- (y);
\end{tikzpicture}
\end{center}
\end{lemma}

\begin{proof} If $c \vee y = z$ and $c \wedge x' = w'$, then $\{x', b, z, c, w', a, y \}$ generates a sublattice of $L$ isomorphic to $L_3$. However, $L$ is semidistributive but $L_3$ is not, so that is impossible. Suppose that $c \vee y < z$ and $c \wedge x' > w'$. As $x' \ngeq y$, $c \wedge x' \ngeq y$. If $c \wedge x' \leq y$, then $w' < c \wedge x' \leq x' \wedge y$. But, as the sublattice of $L$ generated by $\{x', b, z, w', a, y,\}$ is isomorphic to $\textbf{2} \times \textbf{3}$, $w' = x' \wedge y$. Hence, $c \wedge x' \parallel y$. By symmetry, $c \vee y \parallel x'$. Lastly, as the sublattice of $L$ generated by $\{x', b, z, w', a, y,\}$ is isomorphic to $\textbf{2} \times \textbf{3}$, it follows that $b \parallel y$, $x' \parallel a$, $x' \vee y = b \vee (c \vee y)$, and $x' \wedge y = (c \wedge x') \wedge a$. Hence, as $L \in \mathcal{N}$ and $L$ has no doubly reducible elements, $c \wedge x'$, $x'$, $b$, $a$, $y$, and $c \vee y$, depicted below, satisfy the conditions of Lemma \ref{L15}.

\begin{center}
\begin{tikzpicture}[scale=1.2]
\node (x) at (0,0) {$x'$};
\node (b) at (1,1) {$b$};
\node (z) at (2,2) {$z$};
\node (cy) at (3,1) {$c \vee y$};
\node (y) at (4,0) {$y$};
\node (a) at (3,-1) {$a$};
\node (w) at (2,-2) {$w'$};
\node (cx) at (1,-1) {$c \wedge x'$};

\draw (x) -- (b) -- (z) -- (cy) -- (y) -- (a) -- (w) -- (cx) -- (x);
\draw (b) -- (a);
\draw (cx) -- (cy);
\end{tikzpicture}
\end{center}

It follows by Lemma \ref{L15} that $\{x', b, c, a, y \}$ generates a sublattice of $L$ that contains a sublattice isomorphic to $L_{15}$. But that is impossible because $L_{15} \notin \mathcal{N}$. \\

Suppose that $c \vee y < z$ and $c \wedge x' = w'$. This is depicted below.

\begin{center}
\begin{tikzpicture}[scale=0.75]
\node (c) at (1,1) {$(c \vee y) \wedge b$};
\node (x) at (-3,1) {$x'$};
\node (b) at (-1,3) {$b$};
\node (z) at (1,5) {$z$};
\node (w) at (1,-3) {$w'$};
\node (a) at (3,-1) {$a$};
\node (y) at (5,1) {$y$};
\node (cy) at (3,3) {$c \vee y$};
\node (z') at (3,7) {$z'$};
\node (y') at (7,3) {$y'$};
\node (s) at (4,0) {$s$};

\draw (a) -- (s) -- (y);

\draw (x) -- (b) -- (z) -- (cy) -- (y) -- (s) -- (a) -- (w);
\draw (x) -- (w);
\draw (b) -- (c) -- (a);
\draw (z) -- (z') -- (y') -- (y);
\draw (c) -- (cy);
\end{tikzpicture}
\end{center}

Since $a \leq ((c \vee y) \wedge b) \wedge y \leq b \wedge y = a$, $((c \vee y) \wedge b) \wedge y = a$. Because $x' \ngeq y$, $x' \geq c \vee y$. Moreover, as $z = x' \vee y$, $x' \nleq c \vee y$. So $x' \parallel c \vee y$. Furthermore, as $c \leq (c \vee y) \wedge b$, $((c \vee y) \wedge b) \vee y = c \vee y$. Lastly, if $((c \vee y) \wedge b) \wedge x' \leq y$, then $((c \vee y) \wedge b) \wedge x' \leq y \wedge x' = w'$. So if $((c \vee y) \wedge b) \wedge x' > w'$, then as $L$ has no doubly reducible elements, Lemma \ref{L15} implies that $\{((c \vee y) \wedge b) \wedge x', a, x', y, b, c \vee y \}$ generates a sublattice of $L$ that contains a sublattice isomorphic to $L_{15}$. But then, $L_{15}$ is a sublattice of $L$, which is impossible because $L_{15} \notin \mathcal{N}$. \\

Hence, $x' \wedge ((c \vee y) \wedge b) = w'$. So as $x' \vee ((c \vee y) \wedge b) = x' \vee a = b$, the sublattice of $L$ generated by $\{x', b, a, w', (c \vee y) \wedge b, a, y \}$ is isomorphic to the lattice depicted below.
\begin{center}
\begin{tikzpicture}[scale=0.5]
\node (c) at (1,1) {$\circ$};
\node (x) at (-3,1) {$\circ$};
\node (b) at (-1,3) {$\circ$};
\node (z) at (1,5) {$\circ$};
\node (w) at (1,-3) {$\circ$};
\node (a) at (3,-1) {$\circ$};
\node (y) at (5,1) {$\circ$};
\node (cy) at (3,3) {$\circ$};

\draw (x) -- (b) -- (z) -- (cy) -- (y) -- (a) -- (w);
\draw (x) -- (w);
\draw (b) -- (c) -- (a);
\draw (c) -- (cy);
\end{tikzpicture}
\end{center}

We note that $s \parallel b$ for the following reasons. Since $s < y$ and $b \nleq y$, $b \nleq s$. Moreover, $s \leq b$ implies that $s \leq b \wedge y = a$ as $s \leq y$. So, $s \parallel b$. Lastly, recall that $L$ has no doubly reducible elements. \\

If $s \vee b < z$, then by Lemma \ref{L15}, $\{b \vee s, c \vee y, b, y, (c \vee y) \wedge b, s \}$ generates a sublattice $K$ of $L$ such that $K$ contains $L_{15}$ as a sublattice. But that is impossible. \\

If $s \vee b = z$ and $s \vee ((c \vee y) \wedge b) = c \vee y$, then the sublattice of $L$ generated by $\{x', b, z, w', a, s, y, (c \vee y) \wedge b, c \vee y \}$ is isomorphic to $L_9$. However, $L_9 \notin \mathcal{N}$, so that is impossible. \\

If $s \vee b = z$ and if $s \vee ((c \vee y) \wedge b) < c \vee y$, then the sublattice of $L$ generated by $\{b, z, (c \vee y) \wedge b, s \vee ((c \vee y) \wedge b), c \vee y, a, (s \vee ((c \vee y) \wedge b)) \wedge y, y \}$, depicted by the left-most diagram shown below is isomorphic to the lattice depicted by the right-most diagram shown below.

\begin{center}
\begin{tikzpicture}[scale=0.75]
\node (middle) at (1,-1) {$s \vee ((c \vee y) \wedge b)$};
\node (b) at (-3,-1) {$b$};
\node (southwest) at (-1,-3) {$(c \vee y) \wedge b$};
\node (a) at (1,-5) {$a$};
\node (z) at (1,3) {$z$};
\node (northeast) at (3,1) {$c \vee y$};
\node (y) at (5,-1) {$y$};
\node (longest) at (3,-3) {$\;\;\;\;\;\;\;\;\;\;\;\; (s \vee ((c \vee y) \wedge b)) \wedge y$};

\draw (northeast) -- (y);

\draw (b) -- (southwest) -- (a) -- (longest) -- (y) -- (northeast) -- (z);
\draw (b) -- (z);
\draw (southwest) -- (middle) -- (northeast);
\draw (middle) -- (longest);
\end{tikzpicture}
\begin{tikzpicture}[scale=0.75]
\node (c) at (1,-1) {$\circ$};
\node (x) at (-3,-1) {$\circ$};
\node (b) at (-1,-3) {$\circ$};
\node (z) at (1,-5) {$\circ$};
\node (w) at (1,3) {$\circ$};
\node (a) at (3,1) {$\circ$};
\node (y) at (5,-1) {$\circ$};
\node (cy) at (3,-3) {$\circ$};

\draw (x) -- (b) -- (z) -- (cy) -- (y) -- (a) -- (w);
\draw (x) -- (w);
\draw (b) -- (c) -- (a);
\draw (c) -- (cy);
\end{tikzpicture}
\end{center}

Consider the subposet of $L$ depicted below.

\begin{center}
\begin{tikzpicture}[scale=0.75]
\node (middle) at (1,-1) {$s \vee ((c \vee y) \wedge b)$};
\node (b) at (-3,-1) {$b$};
\node (southwest) at (-1,-3) {$(c \vee y) \wedge b$};
\node (a) at (1,-5) {$a$};
\node (z) at (1,3) {$z$};
\node (northeast) at (3,1) {$c \vee y$};
\node (y) at (5,-1) {$y$};
\node (longest) at (3,-3) {$\;\;\;\;\;\;\;\;\;\;\;\; (s \vee ((c \vee y) \wedge b)) \wedge y$};
\node (z') at (3,5) {$z'$};
\node (y') at (7,1) {$y'$};

\draw (northeast) -- (y);

\draw (b) -- (southwest) -- (a) -- (longest) -- (y) -- (northeast) -- (z);
\draw (b) -- (z);
\draw (southwest) -- (middle) -- (northeast);
\draw (middle) -- (longest);
\draw (z) -- (z') -- (y') -- (y);
\end{tikzpicture}
\end{center}

We note the following. Since $b \ngeq y$, $(c \vee y) \wedge b \ngeq y$. Moreover, $(c \vee y) \wedge b \leq y$ implies that $(c \vee y) \wedge b \leq b \wedge y = a$, which is impossible. So $(c \vee y) \wedge b \parallel y$. Similarly, $(c \vee y) \wedge b \parallel y'$. \\

If $(c \vee y) \vee y' < z'$, then, by using the dual of the argument we used to show that $(c \vee y) \wedge b \parallel y$, we see that $(c \vee y) \vee y' \parallel z$ and $(c \vee y) \vee y' \parallel b$. Moreover, $L$ has no doubly reducible elements, and as the sublattice of $L$ generated by $\{b, z, z', a, y, y' \}$ is isomorphic to $\textbf{2} \times \textbf{3}$, it follows that $(c \vee y) \wedge b$, $b$, $z$, $y$, $y'$, and $(c \vee y) \vee y'$ satisfy the conditions of Lemma \ref{L15}. Hence, by Lemma \ref{L15}, $L$ contains a sublattice isomorphic to $L_{15}$. But that is impossible because $L_{15} \notin \mathcal{N}$. \\

If $(c \vee y) \vee y' = z'$, then, as $(c \vee y) \wedge y' = z \wedge y' = y$, the poset depicted above is a sublattice of $L$ that is isomorphic to $L_{12}$. But as $L_{12} \notin \mathcal{N}$, that is impossible. \\

Lastly, suppose that $c \vee y = z$ and that $c \wedge x' > w'$. Since
$$x \wedge a \geq w > w',$$
the dual of the proof for the case when $c \vee y < z$ and $c \wedge x = w$ implies that $L$ contains a sublattice that is isomorphic to $L_{15}$. But that is impossible as $L_{15} \notin \mathcal{N}$.
\end{proof}

We now prove Theorem \ref{partial modularity 2}. 

\begin{proof} Suppose that $(a \vee b_4) \wedge b_5 \neq b_4$ and that $(a \vee b_3) \wedge b_5$ is not covered by $a \vee b_3$. For all $1 \leq i \leq 5$, $(a \vee b_i) \wedge b_5 \geq b_i$. Hence, for all $1 \leq i < j \leq 5$, $a \vee b_i \ngeq (a \vee b_j) \wedge b_5$. Moreover, as $a \nleq b_5$, $a \vee b_i \nleq (a \vee b_j) \wedge b_5$ for all $1 \leq i < j \leq 5$. Hence, for all $1 \leq i < j \leq 5$, $a \vee b_i \parallel (a \vee b_j) \wedge b_5$. It follows that the sublattice of $L$ generated by 
$$\{a \vee b_i : 1 \leq i \leq 5 \} \cup \{(a \vee b_i) \wedge b_5 : 1 \leq i \leq 5 \}$$ is isomorphic to $\textbf{2} \times \textbf{5}$. \\

Now consider the element $((a \vee b_3) \wedge b_5) \vee b_4$. If $(a \vee b_3) \wedge b_5 \geq b_4$, then $a \vee b_3 \geq (a \vee b_3) \wedge b_5 \geq b_4$, but that is impossible as $a \vee b_3 \ngeq b_4$. So $(a \vee b_3) \wedge b_5 \ngeq b_4$. In particular,
$$(a \vee b_3) \wedge b_5 < ((a \vee b_3) \wedge b_5) \vee b_4.$$

If $(a \vee b_3) \wedge b_5 < b_4$, then, as $(a \vee b_4) \wedge b_5 > b_4$,
$$(a \vee b_3) \wedge b_5 < b_4 = ((a \vee b_3) \wedge b_5) \vee b_4 = b_4 < (a \vee b_4) \wedge b_5.$$

If $(a \vee b_3) \wedge b_5 \nless b_4$, then, as $(a \vee b_3) \wedge b_5 \ngeq b_4$, $(a \vee b_3) \wedge b_5 \parallel b_4$. This is depicted below.

\begin{center}
\begin{tikzpicture}[scale = 1.5]
\node (A) at (0,0) {$a \vee b_3$};
\node (B) at (1,1) {$a \vee b_4$};
\node (C) at (2,2) {$a \vee b_5$};
\node (D) at (1,-1) {$(a \vee b_3) \wedge b_5$};
\node (E) at (2,0) {$(a \vee b_4) \wedge b_5$};
\node (F) at (3,1) {$b_5$};
\node (G) at (2,-2) {$b_3$};
\node (H) at (3,-1) {$b_4$};

\draw (B) -- (E) -- (D);
\draw (A) -- (B) -- (C) -- (F) -- (E) -- (H) -- (G) -- (D) -- (A);
\end{tikzpicture}
\end{center}

Because $L$ has no doubly reducible elements, because $(a \vee b_3) \wedge b_5 \parallel b_4$, and because $a \vee b_4 \parallel b_5$, it follows that
$$(a \vee b_3) \wedge b_5 < ((a \vee b_3) \wedge b_5) \vee b_4 < (a \vee b_4) \wedge b_5. $$

Therefore, in either case,
$$(a \vee b_3) \wedge b_5 < ((a \vee b_3) \wedge b_5) \vee b_4 < (a \vee b_4) \wedge b_5. $$
It follows that the subposet depicted below satisfies the conditions of Lemma \ref{12 element subposet}. But, by Lemma \ref{12 element subposet}, that is impossible. 
\begin{center}
\begin{tikzpicture}[scale=1.2]
\node (c) at (1,1) {$c$};
\node (x') at (-4,0) {$a \vee b_1$};
\node (x) at (-2.5,1.5) {$a \vee b_2$};
\node (b) at (-1,3) {$a \vee b_3$};
\node (z) at (1,5) {$a \vee b_4$};
\node (z') at (2.5,6.5) {$a \vee b_5$};
\node (w') at (0,-4) {$(a \vee b_1) \wedge b_5$};
\node (w) at (1.5,-2.5) {$(a \vee b_2) \wedge b_5$};
\node (a) at (3,-1) {$(a \vee b_3) \wedge b_5$};
\node (s) at (4,0) {$\;\;\;\; ((a \vee b_3) \wedge b_5) \vee b_4$};
\node (y) at (5,1) {$(a \vee b_4) \wedge b_5$};
\node (y') at (6.5,2.5) {$b_5$};

\draw (x') -- (x) -- (b) -- (z) -- (z') -- (y') -- (y) -- (s) -- (a) -- (w) -- (w') -- (x');
\draw (x) -- (w);
\draw (b) -- (c) -- (a);
\draw (z) -- (y);
\end{tikzpicture}
\end{center}
This completes the proof of the theorem.
\end{proof}

\begin{remark} In the proof of Theorem \ref{partial modularity 2}, we only used the assumptions that $L \in \mathcal{N}$ and $L$ has no doubly reducible elements. Hence, Theorem \ref{partial modularity 2} describes a property of lattices $L \in \mathcal{N}$ that have no doubly reducible elements.
\end{remark}

\section{Extensions}\label{sec:extensions}

The results of this paper also apply, to a lesser degree, to more general sublattices of free lattices. This is because the lattices $L_i$, for $1 \leq i \leq 15$, from McKenzie's list are connected with other known varieties of lattices. In this section, we explain how the results of this paper can be partially extended to lattices from seven known infinite sequences of semidistributive lattice varieties. \\

Recall that for all lattice $L \in \mathcal{N}$, the dual of $L$ is also in $\mathcal{N}$. Hence, by Theorem \ref{cube theorem}, Theorem \ref{partial modularity 1}, and Theorem \ref{partial modularity 2}, if $L \in \mathcal{N}$ and if $L$ is isomorphic to a sublattice of a free lattice, then $L$ satisfies Theorem \ref{cube theorem}, the dual of Theorem \ref{cube theorem}, Theorem \ref{partial modularity 1}, Theorem \ref{partial modularity 2}, and the dual of Theorem \ref{partial modularity 2}. \\

We now consider more general semidistributive varieties. For all $1 \leq i \leq 15$, let $\mathcal{L}_i$ denote the variety of lattices generated by $L_i$. Rose proved the following in 1984.

\begin{theorem}(Rose, (\cite{VL}, p. 77), \cite{Rose})\label{Roses varieties} Assume that $i \in \{6, 7, 8, 9, 10, 13, 14, 15\}$. Then there exists an infinite sequence 
$$\mathcal{L}_0^i \subset \mathcal{L}_i^1 \subset \mathcal{L}_i^2 \subset \mathcal{L}_i^3 \subset \dots$$ 
of semidistributive lattice varieties such that $\mathcal{L}^i_0 = \mathcal{L}$, and, for all $k \in \mathbb{N}_0$, the following properties hold. The variety $\mathcal{L}^i_{k+1}$ is generated by a finite subdirectly irreducible lattice $L^i_{k+1}$ and $\mathcal{L}^i_{k+1}$ is the unique join-irreducible variety that covers $\mathcal{L}^i_k$.
\end{theorem}

By Theorem \ref{Jonsson Rival}, all of the lattice varieties in Theorem \ref{Roses varieties} are semidistritutive varieties. The proofs of Theorem \ref{cube theorem}, Theorem \ref{partial modularity 1}, and Theorem \ref{partial modularity 2}, only rely on the fact that certain lattices from McKenzie's list $L_i$ for $1 \leq i \leq 15$ are forbidden as sublattices. So, by keeping track of which lattices are forbidden, and by using Theorem \ref{Roses varieties}, we obtain a number of consequences for seven of the eight sequences of semidistributive varieties in Theorem \ref{Roses varieties}.

\begin{corollary}\label{c1} Assume that $i \in \{6, 7, 8\}$, and let $n \in \mathbb{N}_0$. If $L$ is isomorphic to a sublattice of a free lattice and if $L \in \mathcal{L}_i^n$, then $L$ satisfies Theorem \ref{cube theorem}, Theorem \ref{partial modularity 2}, and the dual of Theorem \ref{partial modularity 2}.
\end{corollary}
\begin{proof} If $L$ is isomorphic to a sublattice of a free lattice, if $L \in \mathcal{L}_i^n$, if $n \in \mathbb{N}_0$, and if $i \in \{6, 7, 8\}$, then $L$ does not have a sublattice that is isomorphic to any of the following lattices: $L_9$, $L_{10}$, $L_{11}$, $L_{12}$, $L_{13}$, $L_{14}$, and $L_{15}$. So as the proof of Theorem \ref{cube theorem} relies on $L_{11}$, $L_{12}$, $L_{13}$, $L_{14}$, and $L_{15}$ being forbidden as sublattices, as the proof of Theorem \ref{partial modularity 2} relies on $L_9$, $L_{12}$, and $L_{15}$ being forbidden as sublattices, as the proof of the dual of Theorem \ref{partial modularity 2} relies on $L_{10}$, $L_{11}$, and $L_{15}$ being forbidden as sublattices, and as the proof relies on $\mathcal{N}$ being a semidistributive variety, the corollary follows since $\mathcal{L}_i^n$ is a semidistributive variety.
\end{proof}

\begin{corollary}\label{c2} Let $n \in \mathbb{N}_0$. If $L$ is isomorphic to a sublattice of a free lattice and if $L \in \mathcal{L}_9^n$, then $L$ satisfies Theorem \ref{cube theorem} and the dual of Theorem \ref{partial modularity 2}.
\end{corollary}

\begin{proof} If $L$ is isomorphic to a sublattice of a free lattice, if $n \in \mathbb{N}_0$, and if $L \in \mathcal{L}_9^n$, then $L$ does not have a sublattice that is isomorphic to any of the following lattices: $L_{10}$, $L_{11}$, $L_{12}$, $L_{13}$, $L_{14}$, and $L_{15}$. So as the proof of Theorem \ref{cube theorem} relies on $L_{11}$, $L_{12}$, $L_{13}$, $L_{14}$, and $L_{15}$ being forbidden as sublattices, and as the proof of the dual of Theorem \ref{partial modularity 2} relies on $L_{10}$, $L_{11}$, and $L_{15}$ being forbidden as sublattices, and as the proof relies on $\mathcal{N}$ is a semidistributive variety, the corollary follows since $\mathcal{L}_9^n$ is a semidistributive variety.
\end{proof}

\begin{corollary}\label{c3} Let $n \in \mathbb{N}_0$. If $L$ is isomorphic to a sublattice of a free lattice and if $L \in \mathcal{L}_{10}^n$, then $L$ satisfies Theorem \ref{cube theorem}, the dual of Theorem \ref{cube theorem}, and Theorem \ref{partial modularity 2}.
\end{corollary}

\begin{proof} The proof is dual to the proof of Corollary \ref{c2}.
\end{proof}

\begin{corollary}\label{c4} Let $n \in \mathbb{N}_0$. If $L$ is isomorphic to a sublattice of a free lattice and if $L \in \mathcal{L}_{13}^n$, then $L$ satisfies Theorem \ref{partial modularity 1}, Theorem \ref{partial modularity 2}, the dual of Theorem \ref{partial modularity 2}, and the following property. If $Y$ is an antichain in $L$ such that $|Y| = 3$ and, for some $d \in L$, $a \vee b = d$ for all distinct elements $a,b \in Y$, then at most two elements of $Y$ are not covered by $d$ in $L$. 
\end{corollary}

\begin{proof} The dual of the last property in Corollary \ref{c4} is implied by the first part of the proof of Theorem \ref{cube theorem}. Let $L$ be isomorphic to a sublattice of a free lattice, let $n \in \mathbb{N}_0$, and let $L \in \mathcal{L}_{13}^n$. Then $L$ does not have a sublattice that is isomorphic to any of the following lattices: $L_6$, $L_7$, $L_8$, $L_9$, $L_{10}$, $L_{11}$, $L_{12}$, $L_{14}$, and $L_{15}$. The proof of Theorem \ref{partial modularity 1} relies on $L_6$, $L_7$, $L_8$, $L_9$, $L_{10}$ being forbidden as sublattices, the proof of Theorem \ref{partial modularity 2} relies on $L_9$, $L_{12}$, and $L_{15}$ being forbidden as sublattices, and the proof of the dual of Theorem \ref{partial modularity 2} relies on $L_{10}$, $L_{11}$, and $L_{15}$ being forbidden as sublattices. Moreover, the proof of the first part of Theorem \ref{cube theorem} relies on $L_{12}$, $L_{13}$, and $L_{15}$ being forbidden as sublattices. From this the corollary follows. Since $\mathcal{L}_{13}^n$ is a semidistributive variety.
\end{proof}

\begin{corollary}\label{c5} Let $n \in \mathbb{N}_0$. If $L$ is isomorphic to a sublattice of a free lattice and if $L \in \mathcal{L}_{14}^n$, then $L$ satisfies Theorem \ref{partial modularity 1}, Theorem \ref{partial modularity 2}, the dual of Theorem \ref{partial modularity 2}, and the following property. If $Y$ is an antichain in $L$ such that $|Y| = 3$ and, for some $d \in L$, $a \wedge b = d$ for all distinct elements $a, b \in Y$, then at most two elements of $Y$ cover $d$ in $L$.
\end{corollary}

\begin{proof} The proof is dual to the proof of Corollary \ref{c4}.
\end{proof}


\section*{Acknowledgements}\label{sec:acknow} The author would like to thank Claude Laflamme and Robert Woodrow for their support during the writing of an earlier draft of this paper, and the author would like to thank Stephanie van Willigenburg for her support during the writing of the current paper.

\bibliographystyle{amsplain}

\end{document}